\documentclass[10pt,letterpaper]{article}
\usepackage[utf8]{inputenc}
\usepackage{amsmath}
\usepackage{amsfonts}
\usepackage{amssymb}
\usepackage{graphicx}
\usepackage{verbatim}
\usepackage{mathrsfs}
\usepackage{upref,amsthm,amsxtra,exscale}
\usepackage{cite}
\usepackage[colorlinks=true,urlcolor=blue,
citecolor=red,linkcolor=blue,linktocpage,pdfpagelabels,
bookmarksnumbered,bookmarksopen]{hyperref}
\usepackage{amsthm}
\theoremstyle{plain}

\usepackage{fullpage}

\newtheorem{theorem}{Theorem}[section]
\newtheorem{corollary}[theorem]{Corollary}
\newtheorem{remark}[theorem]{Remark}

\newtheorem{lemma}[theorem]{Lemma}
\newtheorem{proposition}[theorem]{Proposition}

\numberwithin{equation}{section}
\newtheorem*{definition*}{Definition}

\title{Concentration on the Boundary and Sign-Changing Solutions for a Slightly Subcritical Biharmonic Problem}

\author{Salom\'on Alarc\'on$^1$\thanks{S. Alarcón was partially supported by by Fondecyt Projects 1211766 and 1221365.} \\
\texttt{salomon.alarcon@usm.cl}
\and
Jorge Faya$^2$\thanks{J. Faya was supported  by ANID (Chile) through the project FONDECYT/Iniciacion 11190423}\\
  \texttt{jorge.faya@uach.cl} 
\and
Carolina Rey$^1$\\
 \texttt{carolina.reyr@usm.cl}
  }
\date{
	$^1$\small Departamento de Matem\'atica, Universidad T\'ecnica
  Fe\-de\-ri\-co San\-ta Ma\-r\'\i a, Valpara\'\i
  so, Chile.\\
  $^2$\small Instituto de Ciencias F\'isicas y Matem\'aticas, Facultad de Ciencias, Universidad Austral de
Chile, Valdivia, Chile.
}

\begin{document}

\maketitle

\begin{abstract}

We consider the fourth-order nonlinear elliptic problem:
\begin{equation*}
	\left\{
	\begin{array}{ll}
		\Delta(a(x)\Delta u) = a(x) \left\vert u \right\vert^{p-2-\epsilon} u & \text{in } \Omega, \\
		\hspace{0.6cm} u = 0 & \text{on } \partial \Omega, \\
		\hspace{0.6cm} \Delta u = 0 & \text{on } \partial \Omega,
	\end{array}
	\right.  
\end{equation*}
where \(\Omega\) is a smooth, bounded domain in \(\mathbb{R}^N\) with \(N \geq 5\). Here, \(p := \frac{2N}{N-4}\) is the Sobolev critical exponent for the embedding \(H^2 \cap H_0^1(\Omega) \hookrightarrow L^p(\Omega)\), and \(a \in C^2(\overline{\Omega})\) is a strictly positive function on \(\overline{\Omega}\).

We establish sufficient conditions on the function \(a\) and the domain \(\Omega\) for this problem to admit both positive and sign-changing solutions with an explicit asymptotic profile. These solutions concentrate and blow up at a point on the boundary \(\partial \Omega\) as \(\epsilon \to 0\). The proofs of the main results rely on the Lyapunov-Schmidt finite-dimensional reduction method.

\end{abstract}

\section{Introduction}

\subsection{Motivation}

In this paper, we study the almost critical biharmonic problem
\begin{equation}
\label{Prob}
\left\{
\begin{array}{ll}
\Delta(a(x)\Delta u) = a(x)|u|^{p-2-\varepsilon}u & \text{in } \Omega, \\
u = 0 & \text{on } \partial\Omega, \\
\Delta u = 0 & \text{on } \partial\Omega,
\end{array}
\right.
\end{equation}
where \( \Omega \) is a smooth bounded domain in \( \mathbb{R}^N \) (\( N \geq 5 \)), \( a \in C^2(\overline{\Omega}) \) is strictly positive on \( \overline{\Omega} \), \( p := \frac{2N}{N-4} \) is the critical Sobolev exponent for the embedding  \( H^2(\Omega) \cap H^1_0(\Omega) \hookrightarrow L^{p}(\Omega) \), and \( \varepsilon  \) is a positive small parameter.

In general, elliptic equations involving fourth-order operators arise naturally in various mathematical and physical contexts and numerous mathematicians have devoted significant attention to these types of problems (e.g. \cite{Chen,Gupta,Arioli,Chang,Colbois, Davila,Feng,Micheletti,Silva,Wang,WeiYe,Xu}). For example, elliptic fourth-order operators of the form  
\begin{equation*}  
\sum_{i,j,s,h} \frac{\partial^2}{\partial x_i \partial x_j} \left( a_{ijsh}(x) \frac{\partial^2}{\partial x_s \partial x_h} \right)  
\label{eq:fourth-order-operator}  
\end{equation*}  
appear in the study of elastic properties of thin plates. This biharmonic equation has its roots in the work of Sophie Germain, particularly in the Germain-Lagrange equation, a fourth-order partial differential equation describing the behavior of vibrating elastic plates. For more information, see \cite{pastukhova2017operator,bucciarelli2012sophie}. 

Beyond elasticity, the operator on the left-hand side of \eqref{Prob} has gained a lot of attention in the last years, see for example \cite{yan2021classification,caldiroli2016entire,guo2020embeddings,huang2020classification} and references therein.
Moreover,  fourth-order operators like the Paneitz operator play a central role in conformal geometry. These operators are fundamental in questions about finding conformal metrics with constant or prescribed \( Q \)-curvature, analogous to the classical Yamabe problem in lower dimensions. Such problems involve solving fourth-order elliptic PDEs, driven by the Paneitz operator. We refer the reader to \cite{chang2008}.

If \( a(x) \equiv 1 \) and \(\varepsilon=0\), problem \eqref{Prob} becomes
\begin{equation}
\label{eq:a=1}
\left\{
\begin{array}{ll}
\Delta^2 u = |u|^{p-2}u & \text{in } \Omega, \\
u = 0 & \text{on } \partial\Omega, \\
\Delta u = 0 & \text{on } \partial\Omega.
\end{array}
\right.
\end{equation}
This critical problem has been widely studied. Whether it can be solved depends largely on the geometry of $\Omega$.  Indeed, Van der Vorst \cite{Van1} showed that problem \eqref{eq:a=1} does not admit positive solutions when $\Omega$ is star-shaped, whereas Ebobisse and Ahmedou \cite{EboAh} demonstrated that \eqref{eq:a=1} possesses a solution provided that some homology group of $\Omega$ is nontrivial. This topological condition is sufficient, although not essential, as noted by Gazzola, Grunau, and Squassina \cite{gazzola2006existence}, who presented examples of contractible domains where solutions exist. Alarcón and Pistoia studied this problem in domains with circular holes in \cite{alarcon2015paneitz} where they provided a precise description of the asymptotic profile of the solution as the radius of the circular hole approaches zero.

To gain a clearer understanding of the context, it is insightful to compare the Paneitz-type problem \eqref{eq:a=1} with the well-known Yamabe-type problem:
\begin{equation}
\label{eqLaCr}
\begin{cases}
-\Delta u = |u|^{\frac{2N}{N-2}-2}u & \text{in } \Omega, \\
u =  0 & \text{on } \partial\Omega.
\end{cases}
\end{equation}
If the domain \(\Omega\) is star-shaped, Pohozaev’s identity \cite{pokhozhaev1965eigenfunctions} implies that problem  \eqref{eqLaCr} has only the trivial solution. On the other hand, Bahri and Coron \cite{bahri1987nonlinear} demonstrated that \eqref{eqLaCr}  has a solution provided that at least one homology group of \(\Omega\) is nontrivial.  Passaseo showed in \cite{passaseo1995new} that this topological condition is not a necessary requirement for the existence of a solution to \eqref{eqLaCr}.

The almost critical case 
\begin{equation}
	\label{eqLaAlCr}
	\begin{cases}
		-\Delta u = |u|^{\frac{2N}{N-2}-2-\epsilon}u & \text{in } \Omega, \\
		u =  0 & \text{on } \partial\Omega,
	\end{cases}
\end{equation}
with \(\varepsilon\) positive and small enough, has also been widely studied.  The pioneering works of Bahri and Rey \cite{bahri1995variational} and Rey \cite{Rey} established the existence of positive solutions that blow up at one or more points in \(\Omega\) as \(\epsilon \to 0\). Regarding sign-changing solutions, a significant number of solutions exhibiting simple or multiple positive and negative sign-changing solutions were constructed by Bartsch, Micheletti, and Pistoia \cite{BaMiPi}, Musso and Pistoia \cite{MuPi}, and Pistoia and Weth \cite{PiWe}.

In \cite{ackermann2013boundary}, Ackermann, Clapp, and Pistoia studied a more general almost critical problem given by:
\begin{equation}
	\label{eqLaAlCrNon}
	\begin{cases}
		-\text{div}\left(a(x)\nabla u\right) = a(x)|u|^{\frac{2N}{N-2}-2-\varepsilon}u & \text{in } \Omega, \\
		u =  0 & \text{on } \partial\Omega,
	\end{cases}
\end{equation}
where, as before, \(\Omega\) is a bounded smooth domain in \(\mathbb{R}^N\) with \(N \geq 3\), \(\varepsilon > 0\) is a small parameter, and \(a \in C^2(\overline{\Omega})\) is a strictly positive function on \(\overline{\Omega}\). The authors provided sufficient conditions on the function \(a\) and the domain \(\Omega\) for the problem to admit both a positive solution and a sign-changing solution, each of which concentrates and blows up at a point on the boundary \(\partial\Omega\)  as \(\epsilon \to 0\).

The work of Ackermann, Clapp, and Pistoia on the second-order problem \eqref{eqLaAlCrNon} has motivated us to investigate analogous results in the fourth-order setting. In this paper, we focus on the almost critical biharmonic problem \eqref{Prob} and derive sufficient conditions on the function \(a(x)\) and the domain \(\Omega\) that guarantee the existence of solutions with explicit asymptotic profiles as \(\epsilon \to 0\).  As expected, the main challenge arises from the order of the operator. For instance, since it is of fourth order, calculating the error introduced by our ansatz becomes more complex. To compute this error, we need to develop new estimates for the derivatives of the Green's function and its regular part. These estimates allow us to analyze the dominant behavior of the key terms and evaluate their overall contribution.

\medskip

We now describe the profile of the solutions we obtain:
\begin{itemize}
	\item \textit{Boundary blow-up solutions}: These solutions exhibit $k$ multiple blow-up points concentrating at a $k$ different points of the boundary $\partial\Omega$.
	\item \textit{Sign-changing solutions}: These involve two blow-up points of opposite sign collapsing to the same boundary point of the boundary $\partial\Omega$.
\end{itemize}

\medskip

To obtain our results, we employ the Lyapunov-Schmidt finite-dimensional reduction method. The structure of the paper is as follows: In the remainder of this section, we introduce the variational framework and provide a precise description of our main results. In section \ref{sect:Finite-Dim-red} we perform the Lyapunov–Schmidt reduction procedure. Section \ref{Sect:Energy-expansion} provides detailed proofs of the main theorems, including the asymptotic expansions of the energy functional. Section \ref{sect:Green-function} focuses on boundary estimates for the Green's function and its derivatives. We use these results to obtain key estimates. Section \ref{sect:Aprox-solution} investigates the discrepancies between the approximate solutions and those established in the theorems.  

Finally, Section \ref{sect:Error} deals with the estimation of the error term, and Section \ref{sect:Reduced-Energy} with  the expansion of the nonlinear term of the reduced energy functional.

%%%%%%%%%%%%%%%%%%%%%%%%%%%%%%%%%%%%%%%%%%%%%%%%%%%%%%

\subsection{Variational Framework and Main Results}

In this section, we establish the variational framework and present our main results. We begin by introducing the necessary functional spaces: We consider the inner product and corresponding norm on \( H^{2}(\Omega) \cap H_{0}^{1}(\Omega) \), given by  
\[
(u, v) := \int_{\Omega} a(x) \Delta u \Delta v\,dx, \quad \| u \|_{a} := \left( \int_{\Omega} a(x) |\Delta u|^{2} \,dx\right)^{1/2},
\]
which are equivalent to the standard norm and inner product; see Section 2 in \cite{gazzola2010polyharmonic}.  For any \( q \in [1, \infty) \), we define the norm in \( L^{q}(\Omega) \) as  
\[
|u|_{a,q} := \left( \int_{\Omega} a(x) |u|^{q} \,dx\right)^{1/q}.
\]
The Sobolev embedding theorem asserts that \( H^{2}(\Omega) \cap H_{0}^{1}(\Omega) \) is
continuously embedded in \( L^{q}(\Omega) \) for $1 < q \leq \frac{2N}{N-4}$, and this embedding is compact when $q <\frac{2N}{N-4}$.
The embedding \( i: H^{2}(\Omega) \cap H_{0}^{1}(\Omega) \hookrightarrow L^{\frac{2N}{N-4}}(\Omega) \) has an adjoint operator \( i^{*}: L^{\frac{2N}{N+4}}(\Omega) \to H^{2}(\Omega) \cap H_{0}^{1}(\Omega) \), satisfying
\[
(i^{*}(u), \varphi) = \int_{\Omega} a(x) u(x) \varphi(x)\,dx, \quad \forall \varphi \in H^{2}(\Omega) \cap H_{0}^{1}(\Omega).
\]
Equivalently, 
\[
\left\{
\begin{array}{ll}
\Delta(a(x) \Delta i^{*}(u)) = a(x) u(x) & \text{in } \Omega, \\
i^{*}(u) = 0 & \text{on } \partial \Omega, \\
\Delta i^{*}(u) = 0 & \text{on } \partial \Omega.
\end{array}
\right.
\]
Using Hölder's inequality and the Sobolev embedding, we deduce
\[
\| i^{*}(u) \|_{a}^{2} = (i^{*}(u), i^{*}(u)) = \int_{\Omega} a(x) u(x) i^{*}(u)(x) \, dx 
\leq c \| u \|_{\frac{2N}{N+4}} \| i^{*}(u) \|_{a}.
\]
Thus, there exists a constant \( c > 0 \) such that
\[
\| i^{*}(u) \|_{a} \leq c \| u \|_{\frac{2N}{N+4}}, \quad \forall u \in L^{\frac{2N}{N+4}}(\Omega).
\]

We are now in a position to reformulate the Problem \eqref{Prob}.
For \( \ell \geq 0 \), define \( f_{\ell}(s) := |s|^{\frac{8}{N-4} - \ell} s = |s|^{p-2-\ell} s \). Then, we have
\begin{equation}\label{Prob-equiv}
u = i^{*}(f_{\varepsilon}(u)) \quad \text{if and only if } u \text{ solves problem } \eqref{Prob}.
\end{equation}
Now, we introduce the \textit{bubble} \( U_{\delta, \xi} \), which is a function defined by
\[
U_{\delta, \xi}(x) := \alpha_{N} \left( \frac{\delta}{\delta^{2} + |x-\xi|^{2}} \right)^{\frac{N-4}{2}}, \quad x, \xi \in \mathbb{R}^{N}, \, \delta > 0,
\]
where \( \alpha_{N} = (N(N-4)(N-2)(N+2))^{\frac{N-4}{8}} \). These bubbles are the unique positive solutions in \( \mathcal{D}^{2,2}(\mathbb{R}^{N}) \) for the \textit{limit problem}
\[
\Delta^{2} U = U^{\frac{N+4}{N-4}} \quad \text{in } \mathbb{R}^{N}
\]
(see \cite{lin1998classification}).
We use these bubbles as building blocks to construct solutions to \eqref{Prob}. 
The proof of our results relies on the study of the problem that is obtained when linearizing the limit equation around the solution  $U_{\delta, \xi}$. As shown in \cite{lu2000sobolev}, the set of solutions to the linearized equation
\[
\Delta^{2} \psi = f_{0}'(U_{\delta, \xi}) \psi \quad \text{in } \mathbb{R}^{N}, \quad \psi \in \mathcal{D}^{2,2}(\mathbb{R}^{N}),
\]
is an \( (N+1) \)-dimensional linear space spanned by:
\[
\psi_{\delta, \xi}^{0}(x) := \frac{\partial U_{\delta, \xi}}{\partial \delta}(x) = \alpha_{N} \frac{N-4}{2} \delta^{\frac{N-6}{2}} \frac{|x-\xi|^{2} - \delta^{2}}{(\delta^{2} + |x-\xi|^{2})^{\frac{N-2}{2}}},
\]
and, for each \( j = 1, \dots, N \),
\[
\psi_{\delta, \xi}^{j}(x) := \frac{\partial U_{\delta, \xi}}{\partial \xi_{j}}(x) = \alpha_{N} (N-4) \delta^{\frac{N-4}{2}} \frac{x_{j} - \xi_{j}}{(\delta^{2} + |x-\xi|^{2})^{\frac{N-2}{2}}}.
\]
This property is known as the \textit{non-degeneracy} of the solution $U_{\delta, \xi}$.

\medskip
\begin{definition*}[Weak Solutions]
	A function \( u \in H^{2}(\Omega) \cap H_{0}^{1}(\Omega) \) is said to be a weak solution to \eqref{Prob} if, for any test function \( \phi \in H^{2}(\Omega) \cap H_{0}^{1}(\Omega) \), the following holds:
	\[
	\int_{\Omega} \Delta(a(x) \Delta u(x)) \phi(x) \, dx = \int_{\Omega} a(x) |u(x)|^{p-2-\varepsilon} u(x) \phi(x) \, dx.
	\]
\end{definition*}
Hence, problem \eqref{Prob} is the Euler–Lagrange equation for the functional
\begin{equation}\label{Energy}
J_{\varepsilon}:H^{2}(\Omega) \cap H_{0}^{1}(\Omega) \to \mathbb{R}, \quad J_{\varepsilon}(u) = \frac{1}{2} \int_{\Omega} a(x) |\Delta u(x)|^{2} \,dx - \frac{1}{p-\varepsilon} \int_{\Omega} a(x) |u(x)|^{p-\varepsilon} \,dx.
\end{equation}
The functional \( J_{\varepsilon} \) is of class \( C^{2} \), and its derivative is given by
\[
J_{\varepsilon}'(u)[\phi] = \int_{\Omega} a(x) \Delta u (x)\Delta \phi(x) \,dx- \int_{\Omega} a (x)|u(x)|^{p-2-\varepsilon} u(x) \phi(x) \,dx \quad \forall \phi \in H^{2}(\Omega) \cap H_{0}^{1}(\Omega).
\]

\medskip

For any \( u \in \mathcal{D}^{2,2}(\mathbb{R}^{N}) \), we denote by \( Pu \) its projection onto the space \( H^{2}(\Omega) \cap H_{0}^{1}(\Omega) \), which is the unique solution to the problem
\[
\left\{
\begin{array}{ll}
\Delta^{2} Pu = \Delta^{2} u & \text{in } \Omega, \\
Pu = \Delta Pu = 0 & \text{on } \partial \Omega.
\end{array}
\right.
\]

\medskip

We prove the existence of two distinct types of solutions to problem \eqref{Prob}. For the first result, we assume the following condition.
	
	\bigskip
	$(p1)$: There exist $k$ non-degenerate critical points $\zeta_{1}^{0},\dots ,\zeta_{k}^{0}\in\partial\Omega$ of the restriction of $a(x)$ to the boundary $\partial\Omega$, such that:
	\begin{equation*}\label{cond1}
		(\nabla a(\zeta_{j}^{0}),\nu(\zeta_{j}^{0}) )>0 \ \ \ \forall\ j=1\dots,k
	\end{equation*}
	where \(\nu(\zeta_{j}^{0}) \) denotes the inward-pointing unit normal to \( \partial \Omega \) at \( \zeta_{j}^{0} \).
	Under this assumption, we establish the following result:

\begin{theorem}
    \label{thm1}
    Assume that condition \( (p1) \) holds true. Then, there exists \( \varepsilon_{1} > 0 \) such that, for each choice of \( b_{i} \in \{0, 1\} \) and \( 0 < \varepsilon < \varepsilon_{1} \), problem \eqref{Prob} has a solution \( u_{\varepsilon} \) of the form:
    \[
    u_{\varepsilon}(x) = \sum_{j=1}^{k} (-1)^{b_{j}} P U_{\delta_{j, \varepsilon}, \xi_{j, \varepsilon}} (x)+ \phi_{\varepsilon}(x),
    \]
    where the weight \( \delta_{j, \varepsilon} \) of the bubble  \( U_{\delta_{j, \varepsilon}, \xi_{j, \varepsilon}} \) satisfies
        \[
        \delta_{j, \varepsilon} = \varepsilon^{\frac{N-3}{N-4}} d_{j, \varepsilon}, \quad \text{with } d_{j, \varepsilon} \to d_{j} > 0,
        \]
        for \( j = 1, \dots,k \), as \( \varepsilon \to 0 \), and the center \( \xi_{j, \varepsilon} \) of the bubble \( U_{\delta_{j, \varepsilon}, \xi_{j, \varepsilon}} \) satisfies:
        \[
        \xi_{j, \varepsilon} \to \zeta_{j}^{0}, \quad  \text{as } \varepsilon \to 0,
        \]
        for  \( j = 1, \dots,k \), and the rest function $\phi_{\varepsilon}$ is a remainder term.
\end{theorem}

Theorem \ref{thm1} establishes the existence of a solution that may change sign and blows up at the \(k\) points given by  the hypotheses described in $(p1)$. Our next result demonstrates the existence of a sign-changing solution with two blow-up points—one positive and one negative—that collapse to a single point on the boundary of \( \Omega \). To prove this, we assume the following:

\medskip

$(p2)$:  There exists a critical point $\zeta^{0}$ of the restriction of $a(x)$ to the boundary $\partial\Omega$ such that:
\begin{equation*}
	(\nabla a(\zeta^{0}),\nu(\zeta^{0}) )>0 
\end{equation*}
Additionally, we introduce a hypothesis regarding the symmetries of both the domain $\Omega$ and the function $a(x)$ that we describe below:   Assume that there exist a set of vectors \( e_1, \dots, e_{N-1} \in \mathbb{R}^N \), such that the set \( \{ \nu(\zeta^0), e_1, \dots, e_{N-1} \} \) forms an orthonormal basis, such that both \( \Omega \) and \( a(x) \) are invariant under reflections \( R_i \) with respect to each hyperplane \( \zeta_0 + \{ e_i = 0 \} \), i.e.,
\[
R_i(x) \in \Omega \quad \text{and} \quad a(R_i(x)) = a(x) \quad \forall x \in \Omega, \, i = 1, \dots, N-1,
\]
where the reflection \( R_i \) is defined by
	\begin{equation}\label{defRefl}
	R_{i}(x)=\zeta^0+(x,\nu(\zeta^0))\nu(\zeta^0) -(x,e_{i})e_{i}+\sum_{\substack{j=1\\ j \neq i}}^{N-1}(x,e_{j})e_{j}.
\end{equation}	
where, as before, \( \nu := \nu(\zeta^0) \) denotes the inward-pointing unit normal to \( \partial \Omega \) at \( \zeta^0 \). 

We will write 
	\begin{equation}\label{EspSim}
	\mathcal{H}(\Omega) := \{ f \in H^{2}(\Omega) \cap H_{0}^{1}(\Omega) : f(R_{i}(x)) = f(x), \, \forall x \in \Omega, \, i = 1, \dots, N-1 \}.
\end{equation}
for the space of functions on $H^{2}(\Omega) \cap H_{0}^{1}(\Omega)$ that are invariant  under every reflection \( R_i \). Note that if \( \xi = \zeta^{0} + t \nu(\zeta^{0}) \in \Omega \), then $PU_{\delta, \xi}\in \mathcal{H}(\Omega).$

\medskip

Under these assumptions, we establish the following result:

	\begin{theorem}\label{thm2}
    Assume that condition \( (p2) \) holds true. Then, there exists \( \varepsilon_0 > 0 \) such that, for each \( \varepsilon \in (0, \varepsilon_0) \), problem \eqref{Prob} admits a sign-changing solution \( u_{\varepsilon} \) of the form:
    \[
    u_{\varepsilon}(x) = P U_{\delta_{1, \varepsilon}, \xi_{1, \varepsilon}}(x) - P U_{\delta_{2, \varepsilon}, \xi_{2, \varepsilon}}(x) + \phi_{\varepsilon}(x),
    \]
    where the weights \( \delta_{i, \varepsilon} \) of the bubbles \( U_{\delta_{i, \varepsilon}, \xi_{i, \varepsilon}} \) satisfy
        \[
        \varepsilon^{-\frac{N-3}{N-4}} \delta_{i, \varepsilon} \to d_i > 0, \quad \text{as } \varepsilon \to 0,
        \]
        for \( i = 1, 2 \), and the centers \( \xi_{i, \varepsilon} \) of the bubbles satisfy
        \[
        \xi_{i, \varepsilon} = \zeta^{0} + \varepsilon t_{i, \varepsilon} \nu(\zeta^{0}), \quad \text{with } t_{i, \varepsilon} \to t_i > 0, \quad \text{as } \varepsilon \to 0,
        \]
        for \( i = 1, 2 \),  and the rest function $\phi_{\varepsilon}$ is a remainder term. Moreover, the solution \( u_{\varepsilon} \) belongs to the space $\mathcal{H}(\Omega) $.
\end{theorem}

\medskip

The symmetry assumption $(p2)$ helps address technical challenges that arise when looking for a solution whose bubbles collapse to the same point. We believe this hypothesis could be removed, but doing so would require going further in the expansion of the reduced energy functional, making the computations excessively tedious.

\bigskip
 
%%%%%%%%%%%%%%%%%%%%%%%%%%%%%%%%%%%%%%%%%%%%%%%%%%%%%%

\section{Finite Dimensional Reduction}\label{sect:Finite-Dim-red}

In this section, we employ the method of finite-dimensional reduction to transform the original infinite-dimensional problem into a finite-dimensional variational problem. %By focusing on the key parameters governing the concentration phenomena, we establish a framework to construct and characterize the solutions described in Theorems \ref{thm1} and \ref{thm2}.

	The solutions to problem \eqref{Prob} given by Theorem \ref{thm1} are of the form
\begin{equation}\label{typ1}
	u_{\epsilon}=\sum_{j=1}^{k}(-1)^{b_{j}}PU_{\delta_{j},\xi_{j}}+\phi_{\epsilon},
\end{equation}
where $b_{j}\in \{0,1\}$ are fix numbers. The concentration parameters $\delta_{i}$ satisfy
\begin{equation*}
	\delta_{i}=d_{i}\epsilon^{\frac{N-3}{N-4}} \mbox{ for some  } d_{i}>0 \ \mbox{for every } i=1,\dots, k, 
\end{equation*}		
and the concentration points $\xi_{i}$ are given by
\begin{equation*}
	\xi_{i}=\xi_{i}^{0}+\tau_{i}\eta (\xi_{i}^{0})    \mbox{ where } \xi_{i}^{0}\in\partial\Omega \ \ \mbox{and } \tau_{i}=t_{i}\epsilon \mbox{ for some }  t_{i}>0.
\end{equation*}

\medskip

We denote   $\boldsymbol{\xi}=(\xi_{1}^{0},\dots, \xi_{k}^{0})$, $\textbf{d}=(d_{1},\dots, d_{k})$ and $\textbf{t}=(t_{1},\dots, t_{k})$ and we assume that these parameters belong to the set
\begin{equation*}
	\Sigma_{1}:=\{ (\boldsymbol{\xi},\textbf{d}, \textbf{t})\in (\partial\Omega)^{k}\times (0,\infty)^{k}\times (0,\infty)^{k}: \xi_{i}^{0}\neq \xi_{j}^{0} \mbox{ if } i\neq j  \}.
\end{equation*}

\medskip

For each element $(\boldsymbol{\xi},\textbf{d}, \textbf{t})\in \Sigma_{1}$  and every $\epsilon>0$, we consider the fowling decomposition 
\begin{equation*}
	H^{2}(\Omega)\cap H_{0}^{1}(\Omega) =	K^{1,\epsilon}_{(\boldsymbol{\xi},\textbf{d}, \textbf{t})} \oplus K^{1,\epsilon,\perp}_{(\boldsymbol{\xi},\textbf{d}, \textbf{t})}
\end{equation*}
where
\begin{equation*}
	\begin{split}
		K^{1,\epsilon}_{(\boldsymbol{\xi},\textbf{d}, \textbf{t})}&:=span\{ P\psi_{i}^{j}:i=1,\dots,k ,  j=0,1\dots,N \},\\
		K^{1,\epsilon,\perp}_{(\boldsymbol{\xi},\textbf{d}, \textbf{t})}&:=\{u\in H^{2}(\Omega)\cap H_{0}^{1}(\Omega) :(u,P\psi_{i}^{j})=0 ,i=1,\dots,k ,  j=0,1\dots,N \}.
	\end{split}		
\end{equation*}
We consider the respective orthogonal projection operators
\begin{equation*}
	\Pi^{1,\epsilon}_{(\boldsymbol{\xi},\textbf{d}, \textbf{t})}:H^{2}(\Omega)\cap H_{0}^{1}(\Omega)\rightarrow K^{1,\epsilon}_{(\boldsymbol{\xi},\textbf{d}, \textbf{t})} \mbox{ and }  \Pi^{1,\epsilon,\perp}_{(\boldsymbol{\xi},\textbf{d}, \textbf{t})}:H^{2}(\Omega)\cap H_{0}^{1}(\Omega)\rightarrow K^{1,\epsilon,\perp}_{(\boldsymbol{\xi},\textbf{d}, \textbf{t})}
\end{equation*}
defined by
\begin{equation*}
	\Pi^{1,\epsilon}_{(\boldsymbol{\xi},\textbf{d},\textbf{t})}(u):=\sum_{\substack{i=1,\dots,k\\j=0,\dots,N}}(u,P\psi_{i}^{j}) \ \ \mbox{ and } \ \ 	\Pi^{1,\epsilon,\perp}_{(\boldsymbol{\xi},\textbf{d},\textbf{t})}(u):=u-\sum_{\substack{i=1,\dots,k\\j=0,\dots,N}}(u,P\psi_{i}^{j})
\end{equation*}

\vspace{5mm}
On the other hand, the solutions given by Theorem \ref{thm2}  are of the form  
\begin{equation}\label{typ2}
	u_{\epsilon}=PU_{\delta_{1,\epsilon},\xi_{1,\epsilon}}-PU_{\delta_{2,\epsilon},\xi_{2,\epsilon}}+\phi_{\epsilon},
\end{equation}
where the concentration parameters $\delta_{i}$ satisfy:
\begin{equation*}
	\delta_{i,\epsilon}=d_{i}\epsilon^{\frac{N-3}{N-4}} \mbox{ for some  } d_{i}>0 \ \mbox{for each } i=1,\dots, k, 
\end{equation*}		
and the concentration points $\xi_{i}$ are located on the line $L=\{\zeta^{0}+\ell\nu(\zeta^0) : \ell\in\mathbb{R} \}$, i.e. they satisfy
\begin{equation*}
	\xi_{1,\epsilon}=\zeta^{0}+\tau_{1}\nu(\zeta^0), \ \  \xi_{2,\epsilon}=\zeta^{0}+\tau_{2}\nu(\zeta^0)    \mbox{ where } \tau_{i}=t_{i}\epsilon \mbox{ for some }   0<t_{1}<t_{2}.
\end{equation*}
We will write $\textbf{d}=(d_{1}, d_{2})$ and $\textbf{t}=(t_{1} , t_{2})$ and assume that these parameters lie on the set:
\begin{equation*}
	\Sigma_{2}:=\{ (\textbf{d}, \textbf{t})\in (0,\infty)^{2}\times (0,\infty)^{2}: \ 0<t_{1}< t_{2} \}.
\end{equation*}
In this case, we are looking for solutions \( u_{\epsilon} \) in the space \(\mathcal{H}(\Omega)\). For this purpose, the following remark is essential.
\begin{remark}[Symmetry Considerations] \label{RemSim}
	Under assumption \( (p2) \), if \( \xi = \zeta^{0} + t \nu(\zeta^{0}) \in \Omega \), then any solution \( P\psi  \) to the problem
	\begin{equation}\label{pp2}
		\left\{
		\begin{array}{ll}
			\Delta^{2} P\psi  = f_{0}'(U_{\delta, \xi}) \psi  & \text{in } \Omega, \\
			\Delta P\psi  = P\psi  = 0 & \text{on } \partial \Omega,
		\end{array}
		\right.
	\end{equation}
	belongs to \(\mathcal{H}(\Omega)\). This symmetry arises because, if for every \( i = 1, \dots, N-1 \),  we have that \( f_{0}'(U_{\delta, \xi}(R_{i}(x))) = f_{0}'(U_{\delta, \xi}(x)) \), hence if \( P\psi(x) \) is a solution to \eqref{pp2} then \( P\psi(R_{i}(x)) \) is also solution.
\end{remark}

\bigskip

Hence, since we are assuming that \(\Omega\) and \(a\) satisfy \((p2)\), it follows from Remark \ref{RemSim} that, for every $(\zeta^{0}, \textbf{d}, \textbf{t})\in \Sigma_{2}$ the space \(\mathcal{H}(\Omega)\) admits the following decomposition:
\[
\mathcal{H}(\Omega) = K^{2,\epsilon}_{(\zeta^{0}, \textbf{d}, \textbf{t})} \oplus K^{2,\epsilon,\perp}_{(\zeta^{0}, \textbf{d}, \textbf{t})},
\]
where
\[
K^{2,\epsilon}_{(\zeta^{0}, \textbf{d}, \textbf{t})} := \text{span}\{ P\psi_{i}^{j} : i = 1, 2, \ j = 0, 1, \dots, N \} \subseteq \mathcal{H}(\Omega),
\]
and its respective orthogonal complement in \(\mathcal{H}(\Omega)\) is defined as:
\[
K^{2,\epsilon,\perp}_{(\zeta^{0}, \textbf{d}, \textbf{t})} := \{ u \in \mathcal{H}(\Omega) : (u, P\psi_{i}^{j}) = 0 \text{ for } i = 1, \dots, k \text{ and } j = 0, 1, \dots, N \}.
\]
Then we introduce the orthogonal projection operators:
\begin{equation*}
	\Pi^{2,\epsilon}_{(\zeta^{0}, \textbf{d}, \textbf{t})}:\mathcal{H}(\Omega) \rightarrow K^{2,\epsilon}_{(\zeta^{0}, \textbf{d}, \textbf{t})} \mbox{ and }  \Pi^{2,\epsilon,\perp}_{(\zeta^{0}, \textbf{d}, \textbf{t})}:\mathcal{H}(\Omega) \rightarrow K^{2,\epsilon,\perp}_{(\zeta^{0}, \textbf{d}, \textbf{t})}.
\end{equation*}

\medskip

For simplicity, when there is no danger of confusion, we will write $(\zeta^{0}, \textbf{d}, \textbf{t})=(\boldsymbol{\xi},\textbf{d}, \textbf{t})$. Also,  for each $\epsilon>0$ and any element $(\boldsymbol{\xi},\textbf{d}, \textbf{t})\in \Sigma_{i}$ we will write: 
\begin{equation*}
	V^{\epsilon}_{(\boldsymbol{\xi},\textbf{d}, \textbf{t})}=V_{(\boldsymbol{\xi},\textbf{d}, \textbf{t})}:=	\sum_{j=1}^{k}(-1)^{b_{j}}PU_{\delta_{j},\xi_{j}} \ \ \mbox{ and }\ \  	V^{\epsilon}_{(\boldsymbol{\xi},\textbf{d}, \textbf{t})}=	V_{(\textbf{d}, \textbf{t})}=PU_{\delta_{1},\xi_{1}}-PU_{\delta_{2},\xi_{2}}.
\end{equation*}
We will also write: $K^{\epsilon,\perp}_{(\boldsymbol{\xi},\textbf{d}, \textbf{t})}:=	K^{i,\epsilon,\perp}_{(\boldsymbol{\xi},\textbf{d}, \textbf{t})}$, $\Pi^{\epsilon,\perp}_{(\boldsymbol{\xi},\textbf{d}, \textbf{t})}=\Pi^{i,\epsilon,\perp}_{(\boldsymbol{\xi},\textbf{d}, \textbf{t})}$ and $\Sigma=\Sigma_{i}$ for $i=1,2$.

\newpage

In both cases \eqref{typ1} and \eqref{typ2}, the correction term \( \phi_{\varepsilon} \) belongs to \( K_{(\boldsymbol{\xi}, \mathbf{d}, \mathbf{t})}^{\varepsilon, \perp} \). Solving problem \eqref{Prob-equiv} reduces to finding an element \( (\boldsymbol{\xi}, \mathbf{d}, \mathbf{t}) \in \Sigma \) and \( \phi_{\varepsilon} \in K_{(\boldsymbol{\xi}, \mathbf{d}, \mathbf{t})}^{\varepsilon, \perp} \) that satisfies the following system
	
	\begin{equation}\label{pa}
		\Pi^{\varepsilon}_{(\boldsymbol{\xi},\textbf{d}, \textbf{t})}(V_{(\boldsymbol{\xi},\textbf{d}, \textbf{t})}+\phi_{\varepsilon}-i^{*}(f_{\varepsilon}(V_{(\boldsymbol{\xi},\textbf{d}, \textbf{t})}+\phi_{\varepsilon})))=0,
	\end{equation}
	and 
	\begin{equation}\label{paort}
		\Pi^{\varepsilon,\perp}_{(\boldsymbol{\xi},\textbf{d}, \textbf{t})}(V_{(\boldsymbol{\xi},\textbf{d}, \textbf{t})}+\phi_{\varepsilon}-i^{*}(f_{\varepsilon}(V_{(\boldsymbol{\xi},\textbf{d}, \textbf{t})}+\phi_{\varepsilon})))=0.
	\end{equation}
	
	\bigskip
	To achieve this goal, we first show that for each $(\boldsymbol{\xi},\textbf{d}, \textbf{t})\in \Sigma$  and for any small $\varepsilon>0$ we can find an element  $\phi_{\boldsymbol{\xi},\textbf{d}, \textbf{t}}^{\varepsilon} \in K^{\varepsilon,\perp}_{(\boldsymbol{\xi},\textbf{d}, \textbf{t})}$  that satisfies $\eqref{paort}$.  This is the essence of Remark \ref{remarkphi}. To this aim we define a linear operator $L^{\varepsilon}_{(\boldsymbol{\xi},\textbf{d}, \textbf{t})} : K^{\varepsilon,\perp}_{(\boldsymbol{\xi},\textbf{d}, \textbf{t})} \to K^{\varepsilon,\perp}_{(\boldsymbol{\xi},\textbf{d}, \textbf{t})}$ by
	\begin{equation*}
		L^{\varepsilon}_{(\boldsymbol{\xi},\textbf{d}, \textbf{t})} \phi=	\Pi^{\varepsilon,\perp}_{(\boldsymbol{\xi},\textbf{d}, \textbf{t})}(\phi-i^{*}(f'_{\varepsilon}(V_{(\boldsymbol{\xi},\textbf{d}, \textbf{t})})\phi)).
	\end{equation*}
	
	The following statement holds true.
	
	\begin{proposition}
		For any compact subset $C$ of $\Sigma$ there exist $\varepsilon_0 > 0$ and $c > 0$ such that for each $\varepsilon \in (0, \varepsilon_0)$ and $(\boldsymbol{\xi},\textbf{d}, \textbf{t}) \in C$ the operator $L^{\varepsilon}_{(\boldsymbol{\xi},\textbf{d}, \textbf{t})}$ is invertible and
		\[
		\|L^{\varepsilon}_{(\boldsymbol{\xi},\textbf{d}, \textbf{t})} \phi\| \geq c \|\phi\| \quad \forall \phi \in K^{\perp}_{(\boldsymbol{\xi},\textbf{d}, \textbf{t})}.
		\]
	\end{proposition}
	
	\begin{proof}
The proof of this proposition is standard and, with the necessary modifications, can be established by following, for example, the approach outlined in \cite{musso2002multispike}.
	\end{proof}

	\bigskip
	From the previous proposition, by using the inverse of the operator  $L^{\varepsilon}_{(\boldsymbol{\xi},\textbf{d}, \textbf{t})} $, for any sufficiently small $\varepsilon>0$ we can define the operator $T^{\varepsilon}_{(\boldsymbol{\xi},\textbf{d}, \textbf{t})} : K^{\varepsilon,\perp}_{(\boldsymbol{\xi},\textbf{d}, \textbf{t})} \to K^{\varepsilon,\perp}_{(\boldsymbol{\xi},\textbf{d}, \textbf{t})}$ by
	\begin{equation*}
		T^{\varepsilon}_{(\boldsymbol{\xi},\textbf{d}, \textbf{t})} \phi =L^{\varepsilon,-1}_{(\boldsymbol{\xi},\textbf{d}, \textbf{t})}\Pi^{\varepsilon,\perp}_{(\boldsymbol{\xi},\textbf{d}, \textbf{t})}(-V_{(\boldsymbol{\xi},\textbf{d}, \textbf{t})}-i^{*}(f'_{\varepsilon}(V_{(\boldsymbol{\xi},\textbf{d}, \textbf{t})})\phi)+i^{*}(f_{\varepsilon}(V_{(\boldsymbol{\xi},\textbf{d}, \textbf{t})}+\phi))).
	\end{equation*}
	
	\bigskip
	
	\begin{proposition}
		For any compact subset $C$ of $\Sigma$ there exist $\varepsilon_0 > 0$, $c >0$  and  $\sigma>0$ such that for each $\varepsilon \in (0, \varepsilon_0)$ and $(\boldsymbol{\xi},\textbf{d}, \textbf{t}) \in C$ the operator $T^{\varepsilon}_{(\boldsymbol{\xi},\textbf{d}, \textbf{t})} $ is a contraction map and
		\[
		\lVert T^{\varepsilon}_{(\boldsymbol{\xi},\textbf{d}, \textbf{t})} \phi \rVert \leq  c\varepsilon^{\frac{1}{2}+\sigma} \quad \forall \phi \in K^{\varepsilon,\perp}_{(\boldsymbol{\xi},\textbf{d}, \textbf{t})}.
		\]
	\end{proposition}
	\begin{proof}
		In Lemma \ref{error},  we show that there exists $\sigma>o$ such that
		\[
		\|E_{(\boldsymbol{\xi},\textbf{d}, \textbf{t})}\| = O\left( \varepsilon^{\frac{1}{2} + \sigma} \right).
		\]
		where 
		\begin{equation}\label{Error}
			E_{(\boldsymbol{\xi},\textbf{d}, \textbf{t})}=:  \Delta(a(x)\Delta V_{(\boldsymbol{\xi},\textbf{d}, \textbf{t})}) - a(x)|V_{(\boldsymbol{\xi},\textbf{d}, \textbf{t})}|^{p-2-\varepsilon}V_{(\boldsymbol{\xi},\textbf{d}, \textbf{t})}
		\end{equation}
		is the error term. From here, we proceed as in \cite{bartsch2006existence}.
	\end{proof}
	
	\bigskip

	\begin{remark}\label{remarkphi}
		Therefore, under the conditions described above, the operator $T^{\varepsilon}_{(\boldsymbol{\xi},\textbf{d}, \textbf{t})}$ has a unique fixed point, which we denote by $\phi_{(\boldsymbol{\xi},\textbf{d}, \textbf{t})}^{\varepsilon} \in K^{\varepsilon,\perp}_{(\boldsymbol{\xi},\textbf{d}, \textbf{t})}$, and it satisfies:
		\[
		\lVert\phi_{(\boldsymbol{\xi},\textbf{d}, \textbf{t})}^{\varepsilon}  \rVert \leq  c\varepsilon^{\frac{1}{2}+\sigma}.
		\]
		Hence, we can write
		\begin{equation*}
			L^{\varepsilon}_{(\boldsymbol{\xi},\textbf{d}, \textbf{t})} \phi_{(\boldsymbol{\xi},\textbf{d}, \textbf{t})}^{\varepsilon}  =\Pi^{\varepsilon,\perp}_{(\boldsymbol{\xi},\textbf{d}, \textbf{t})}\left(-V_{(\boldsymbol{\xi},\textbf{d}, \textbf{t})}-i^{*}(f'_{\varepsilon}(V_{(\boldsymbol{\xi},\textbf{d}, \textbf{t})})\phi_{(\boldsymbol{\xi},\textbf{d}, \textbf{t})}^{\varepsilon} )+i^{*}\left(f_{\varepsilon}\left(V_{(\boldsymbol{\xi},\textbf{d}, \textbf{t})}+\phi_{(\boldsymbol{\xi},\textbf{d}, \textbf{t})}^{\varepsilon} \right)\right)\right)
		\end{equation*}
		which shows that  the element  $\phi_{(\boldsymbol{\xi},\textbf{d}, \textbf{t})}^{\varepsilon} $  satisfies equation \eqref{paort}
		Moreover, that the map $$(\boldsymbol{\xi},\textbf{d}, \textbf{t}) \mapsto \phi_{(\boldsymbol{\xi},\textbf{d}, \textbf{t})}^{\varepsilon} $$ is of class $C^1$.
	\end{remark}
	Since  $\phi_{(\boldsymbol{\xi},\textbf{d}, \textbf{t})}^{\varepsilon} $  satisfies equation \eqref{paort}, we have that:
	\begin{equation*}
		\begin{aligned}
			&V_{(\boldsymbol{\xi},\textbf{d}, \textbf{t})}+\phi_{(\boldsymbol{\xi},\textbf{d}, \textbf{t})}^{\varepsilon}-i^{*}\left(f_{\varepsilon}\left(V_{(\boldsymbol{\xi},\textbf{d}, \textbf{t})}+\phi_{(\boldsymbol{\xi},\textbf{d}, \textbf{t})}^{\varepsilon}\right)\right)\\
			=&\sum_{\substack{i=1,\dots,k\\j=0,\dots,N}}\left( V_{(\boldsymbol{\xi},\textbf{d}, \textbf{t})}+\phi_{(\boldsymbol{\xi},\textbf{d}, \textbf{t})}^{\varepsilon}-i^{*}\left(f_{\varepsilon}\left(V_{(\boldsymbol{\xi},\textbf{d}, \textbf{t})}+\phi_{(\boldsymbol{\xi},\textbf{d}, \textbf{t})}^{\varepsilon}\right)\right),P\psi_{i}^{j}\right) 
		\end{aligned}
	\end{equation*}
	Hence, $V_{(\boldsymbol{\xi},\textbf{d}, \textbf{t})}+\phi_{(\boldsymbol{\xi},\textbf{d}, \textbf{t})}^{\varepsilon}$ is a solution to \eqref{Prob} if  and only if:
	\begin{equation}\label{consol}
		\begin{aligned}
			\sum_{\substack{i=1,\dots,k\\j=0,\dots,N}}\left( V_{(\boldsymbol{\xi},\textbf{d}, \textbf{t})}+\phi_{(\boldsymbol{\xi},\textbf{d}, \textbf{t})}^{\varepsilon}-i^{*}\left(f_{\varepsilon}\left(V_{(\boldsymbol{\xi},\textbf{d}, \textbf{t})}+\phi_{(\boldsymbol{\xi},\textbf{d}, \textbf{t})}^{\varepsilon}\right)\right),P\psi_{i}^{j}\right) =0
		\end{aligned}
	\end{equation}
	\bigskip
Recall that, the solutions to problem \eqref{Prob} can be characterized as the critical points of the functional $J_{\varepsilon}$. We define the reduced energy functional \( \mathcal{I}_\varepsilon^{i} : \Sigma_{i} \rightarrow \mathbb{R} \) as
\[
\mathcal{I}^{i}_\varepsilon(\boldsymbol{\xi}, \textbf{d}, \textbf{t}) = J_\varepsilon\left(V_{(\boldsymbol{\xi}, \textbf{d}, \textbf{t})}^\varepsilon + \phi_{(\boldsymbol{\xi}, \textbf{d}, \textbf{t})}^\varepsilon\right),
\]
where, depending on the case:
- If \( (\boldsymbol{\xi}, \textbf{d}, \textbf{t}) \in \Sigma_{1} \), \( J_{\varepsilon} \) is given by \eqref{Energy}.
- If \( (\zeta^0, \textbf{d}, \textbf{t}) \in \Sigma_{2} \), \( J_{\varepsilon} \) is defined as
\begin{equation}
	J_{\varepsilon} : \mathcal{H}(\Omega) \to \mathbb{R}, \quad J_{\varepsilon}(u) = \frac{1}{2} \int_{\Omega} a(x) |\Delta u(x)|^{2} \, dx- \frac{1}{p-\varepsilon} \int_{\Omega} a(x) |u(x)|^{p-\varepsilon} \,dx.
\end{equation}

 The next lemma, following from Remark \ref{remarkphi} and equation \eqref{consol}, reduces the problem of finding solutions to \eqref{Prob} to the problem of finding critical points of the reduced energy functional $\mathcal{I}_\varepsilon $.
	\begin{lemma}\label{reduc}
		The element $ (\boldsymbol{\xi},\textbf{d}, \textbf{t}) \in \Sigma $ is a critical point of $\mathcal{I}_\varepsilon$ if and only if the function \( u_\varepsilon = V_{(\boldsymbol{\xi},\textbf{d}, \textbf{t})}^\varepsilon + \phi_{(\boldsymbol{\xi},\textbf{d}, \textbf{t})}^\varepsilon \) is a critical point of the functional \( J_\varepsilon \).
	\end{lemma}
	
\begin{proof}
The proof of this proposition follows standard arguments and can be obtained, with suitable modifications, by using the approach outlined in \cite{bahri1995variational}. We omit it here.
\end{proof}

%%%%%%%%%%%%%%%%%%%%%%%%%%%%%%%%%%%%%%%%%%%%%%%%%

	\bigskip\bigskip
	
	\section{Energy functional expansion and proofs of the main Theorems}\label{Sect:Energy-expansion}
    %Expansión del funcional reducido y Prueba de los Teoremas}
    
   In this section, we derive an asymptotic expansion of the reduced energy functional associated with problem \eqref{Prob} and provide the proofs of the main theorems, \ref{thm1} and \ref{thm2}.

	\subsection{The proof of  Theorem \ref{thm1}}
	
	The proof of Theorem \ref{thm1} is thus reduced to showing the existence of a critical point of the reduced energy functional $\mathcal{I}^{1}_\varepsilon:\Sigma_{1}\rightarrow \mathbb{R}$ defined by
	$$\mathcal{I}^{1}_\varepsilon(\boldsymbol{\xi},\textbf{d}, \textbf{t})=J_\varepsilon(V_{(\boldsymbol{\xi},\textbf{d}, \textbf{t}) }^\varepsilon + \phi_{(\boldsymbol{\xi},\textbf{d}, \textbf{t}) }^\varepsilon)=J_{\varepsilon}\left( \sum_{i=1}^{k}(-1)^{b_{i}}PU_{\delta_{i},\xi_{i}}++ \phi_{(\boldsymbol{\xi},\textbf{d}, \textbf{t}) }^\varepsilon\right) .$$ 
	
	We will compute an asymptotic expansion of $\mathcal{I}^{1}_\varepsilon$.
	\begin{proposition}\label{exp1}
		The asymptotic expansion
		\[
		\mathcal{I}^{1}_\varepsilon(\boldsymbol{\xi},\textbf{d}, \textbf{t})=\sum_{i=1}^{k}a(\xi_{i}^{0})\gamma_{1}\left( \frac{p-2}{2p}-\varepsilon\log (\varepsilon)\frac{N-3}{2p}\right) +\varepsilon F^{1}(\boldsymbol{\xi},\textbf{d}, \textbf{t})+o(\varepsilon)
		\]
		holds true $C^{1}$-uniformly on compact subsets of $\Sigma_{1}$. Here the
		function $F^{1}:\Sigma_{1}\rightarrow \mathbb{R}$ is given by%
		\begin{equation*}
			\begin{aligned}
				F^{1}(\boldsymbol{\xi},\textbf{d}, \textbf{t})&:=\sum_{i=1}^{k}  a(\xi_{i}^{0})\frac{p\gamma_{3}-\gamma_{1}}{p^{2}}\\
				&+\sum_{i=1}^{k}\left(\frac{p-2}{2p}t_{i}(\nabla a(\xi_{i}^{0})\cdot\eta (\xi_{i}^{0}) )\gamma_{1}+\frac{a(\xi_{i}^{0})}{2}\left( \frac{d_{i}}{2t_{i}}\right) ^{N-4}\gamma_{2}-a(\xi_{i}^{0})\frac{N-4}{2p}\log(d_{i})\gamma_{1}\right)
			\end{aligned}
		\end{equation*}
	\end{proposition}
	
	\begin{proof}
		We have that
		\begin{equation*}
			\begin{aligned}
				J(	V_{(\boldsymbol{\xi},\textbf{d}, \textbf{t})}):&=\frac{1}{2} \int_{\Omega}a(x)\Delta\left(	V_{(\boldsymbol{\xi},\textbf{d}, \textbf{t})} (x)\right)^{2} \,dx-\frac{1}{p-\varepsilon} \int_{\Omega}a(x)\left|	V_{(\boldsymbol{\xi},\textbf{d}, \textbf{t})}(x)\right|^{p-\varepsilon}\,dx.
			\end{aligned}
		\end{equation*}
		The term 	
		\begin{equation*}
			\frac{1}{p-\varepsilon} \int_{\Omega}a(x)\left|	V_{(\boldsymbol{\xi},\textbf{d}, \textbf{t})}(x)\right|^{p-\varepsilon}\, dx=	\frac{1}{p-\varepsilon} \int_{\Omega}a(x)|\sum_{i=1}^{k}(-1)^{\lambda_{i}}PU_{i}(x)|^{p-\varepsilon} \, dx,
		\end{equation*}
		is estimated in Lemma \ref{lemma:NL}. Next we give the expansion of the remaining term. 	Notice that:
		\begin{equation*}
			\begin{aligned}
				&\int_{\Omega}a(x)\Delta\left(	V_{(\boldsymbol{\xi},\textbf{d}, \textbf{t})}(x)\right)^{2}\,dx=	 \int_{\Omega}a(x)(\Delta(\sum_{i=1}^{k}(-1)^{b_{i}}PU_{\delta_{i},\xi_{i}}(x)))^{2}\,dx\\
				&=	 \sum_{i=1}^{k}\int_{\Omega}a(x)\left( \Delta(PU_{\delta_{i},\xi_{i}}(x))\right)^{2} \,dx+2\sum_{i\neq j}\sum_{i=1}^{k}(-1)^{b_{i}}(-1)^{b_{j}}\int_{\Omega}a(x)\Delta(PU_{\delta_{i},\xi_{i}}(x))\Delta(PU_{\delta_{j},\xi_{j}}(x))\,dx
			\end{aligned}
		\end{equation*}
		A direct calculation, for any $1\leq i\leq j \leq k$, shows that:
		\begin{equation*}
			\begin{aligned}
				\int_{\Omega}a(x)\Delta(PU_{\delta_{i},\xi_{i}}(x))\Delta(PU_{\delta_{j},\xi_{j}}(x))\,dx
				&=\int_{\Omega}\Delta a(x)\Delta\left(PU_{\delta_{i},\xi_{i}}\right)PU_{\delta_{j},\xi_{j}}\,dx+2\int_{\Omega}\nabla a(x)\nabla\left( \Delta(PU_{\delta_{i},\xi_{i}})\right)PU_{\delta_{j},\xi_{j}}\,dx\\
				&+\int_{\Omega}a(x)\Delta^{2}(PU_{\delta_{i},\xi_{i}})PU_{\delta_{j},\xi_{j}}\,dx\\
			\end{aligned}
		\end{equation*}
		Using Proposition \ref{est1}, for any $1\leq i\leq j \leq k$, we obtain:
		\begin{equation}\label{ql1}
			\int_{\Omega}\Delta a(x)\Delta(PU_{\delta_{i},\xi_{i}})PU_{\delta_{j},\xi_{j}}\,dx=o(\varepsilon) \mbox{ and }	\int_{\Omega}\nabla a(x)\nabla\left( \Delta(PU_{\delta_{i},\xi_{i}})\right)PU_{\delta_{j},\xi_{j}}\,dx=o(\varepsilon)
		\end{equation}
		Therefore, in the case where $i\neq j$, from equation \eqref{inj} in Proposition \ref{propij}, we obtain 
		\begin{equation}
			\begin{aligned}
				\int_{\Omega}a(x)\Delta(PU_{\delta_{i},\xi_{i}})\Delta(PU_{\delta_{j},\xi_{j}})\,dx=o(\varepsilon)
			\end{aligned}
		\end{equation}
		For the other case,  by Proposition \ref{tmi}, we obtain:
		\begin{equation}\label{ql2}
			\begin{aligned}
				&\int_{\Omega}a(x)\Delta^{2}(PU_{\delta_{i},\xi_{i}})PU_{\delta_{i},\xi_{i}}\,dx=	\int_{\Omega}a(x)U_{\delta_{i},\xi_{i}}^{\frac{N+4}{N-4}}PU_{\delta_{i},\xi_{i}}\,dx\\
				&=		\int_{\Omega}a(x)U_{\delta_{i},\xi_{i}}^{\frac{2N}{N-4}} \,dx+\int_{\Omega}a(x)U_{\delta_{i},\xi_{i}}^{\frac{N+4}{N-4}}\left(PU_{\delta,\xi}(x)-U_{\delta,\xi}(x) \right)\,dx \\
				&=a(\xi_{i}^{0})\gamma_{1}+\varepsilon t_{i}\nabla a(\xi_{i}^{0})\cdot\nu(\xi_{i}^{0})\gamma_{1}
				-\varepsilon\left( \frac{d_{i}}{2t_{i}}\right) ^{N-4}a(\xi_{i}^{0})\gamma_{2}+o(\varepsilon^{1+\sigma}).
			\end{aligned}
		\end{equation}
		Therefore, from the equations \eqref{ql1} and \eqref{ql2}, it follows:
		\begin{equation*}
			\begin{aligned}
				&\frac{1}{2} \int_{\Omega}a(x)\Delta(	V_{(\boldsymbol{\xi},\textbf{d}, \textbf{t})})^{2}\,dx\\
				&=\frac{1}{2}\sum_{i=1}^{k}\left(  a(\xi_{i}^{0})\gamma_{1}+ \varepsilon t_{i}\nabla a(\xi_{i}^{0})\cdot\nu(\xi_{i}^{0})\gamma_{1}-\varepsilon\left( \frac{d_{i}}{2t_{i}}\right) ^{N-4}a(\xi_{i}^{0})\gamma_{2}\right) +o(\varepsilon^{1+\sigma}).
			\end{aligned}
		\end{equation*}
	\end{proof}
	
\bigskip

	\begin{proof}{of Theorem \eqref{thm1}}
		
		By Proposition \ref{exp1} we have that 
		$$
		\mathcal{I}^{1}_\varepsilon(\boldsymbol{\xi},\textbf{d}, \textbf{t})=\sum_{i=1}^{k}a(\xi_{i}^{0})\gamma_{1}\left( \frac{p-1}{2(p+1)}-\varepsilon\log (\varepsilon)\frac{n-3}{2(p+1)}\right) +O(\varepsilon),$$
		$C^{1}$-uniformly on compact subsets of $\Sigma_{1}$. Since we are assuming that   $\zeta_{1}^{0},\dots ,\zeta_{k}^{0}\in\partial\Omega$ are  nondegenerated critical points of the function $a_{|\partial\Omega}$  we have that, for $\varepsilon$ small enough, there exists $\boldsymbol{\zeta_{\varepsilon}}=(\zeta_{\varepsilon,1}^{0},\dots , \zeta_{\varepsilon,k}^{0})\in(\partial\Omega)^{k}$, such that $\nabla_{\boldsymbol{\xi}}\mathcal{I}^{1}_\varepsilon(\boldsymbol{\zeta_{\varepsilon}},\textbf{d}, \textbf{t})=0$ and $\zeta_{i,\varepsilon}^{0}\rightarrow \zeta_{i}^{0}$ as $\varepsilon$ goes to zero for every $i=1,\dots,k$.
		
		Proposition \ref{exp1} also implies that
		\begin{equation}\label{simpl1}
			\begin{aligned}
				\mathcal{I}^{1}_\varepsilon(\boldsymbol{\zeta_{\varepsilon}},\textbf{d}, \textbf{t})&=\sum_{i=1}^{k}a(\zeta_{\varepsilon,i}^{0})\gamma_{1}\left( \frac{p-1}{2(p+1)}-\varepsilon\log (\varepsilon)\frac{n-3}{2(p+1)}\right) +\varepsilon F^{1}(\boldsymbol{\zeta_{\varepsilon}},\textbf{d}, \textbf{t})+o(\varepsilon)\\
				&=\sum_{i=1}^{k}a(\zeta_{i}^{0})\gamma_{1}\left( \frac{p-1}{2(p+1)}-\varepsilon\log (\varepsilon)\frac{n-3}{2(p+1)}\right) +\varepsilon F_{0}^{1}(\textbf{d}, \textbf{t})+o(\varepsilon)
			\end{aligned}
		\end{equation}
		where $F^{1}_{0}:\{ (\textbf{d}, \textbf{t})\in (0,\infty)^{k}\times (0,\infty)^{k} \}\rightarrow \mathbb{R}$ is defined by 
		\begin{equation*}
			\begin{aligned}
				F^{1}_{0}(\textbf{d}, \textbf{t})&:=\sum_{i=1}^{k}\left(  a(\zeta_{i}^{0})\frac{(p+1)\gamma_{3}-\gamma_{1}}{(p+1)^{2}}+\frac{p-1}{2(p+1)}t_{i}(\nabla a(\zeta_{i}^{0})\cdot\eta (\zeta_{i}^{0}) )\gamma_{1}\right)\\
				&+\sum_{i=1}^{k}\left(\frac{a(\zeta_{i}^{0})}{2}\left( \frac{d_{i}}{2t_{i}}\right) ^{n-4}\gamma_{2}-a(\zeta_{i}^{0})\frac{n-4}{2(p+1)}\log(d_{i})\gamma_{1}\right)
			\end{aligned}
		\end{equation*}
		Since $(\nabla a(\zeta_{i}^{0})\cdot\eta (\zeta_{i}^{0}))>0$ for every $i=1,\dots,k$, see \eqref{cond1}, the function $F^{1}_{0}$ has a critical point which is a strict minimum and stable. Hence, for $\varepsilon$ is small enough, it follows from equation \eqref{simpl1} that  there exists $(\textbf{d}_{\varepsilon}, \textbf{t}_{\varepsilon})$ such that $\nabla_{(\textbf{d}, \textbf{t})}\mathcal{I}^{1}_\varepsilon(\boldsymbol{\zeta^{0}_{\varepsilon}},\textbf{d}_{\varepsilon}, \textbf{t}_{\varepsilon})=0$. We conclude that the functional  $\mathcal{I}^{1}_\varepsilon$ has a critical point. Finally,  the proof of Theorem \ref{thm1}  follows from Lemma \ref{reduc}. 
	\end{proof}

	\subsection{The proof of Theorem \ref{thm2}}

	The proof of Theorem \ref{thm2} is thus reduced to show the existence of a critical point of the reduced energy functional $\mathcal{I}^{2}_\varepsilon:\Sigma_{2}\rightarrow \mathbb{R}$ defined by
	$$\mathcal{I}^{2}_\varepsilon(\zeta^{0},\textbf{d}, \textbf{t})=\mathcal{I}^{2}_\varepsilon(\textbf{d}, \textbf{t})=J_{\varepsilon}\left(	V_{(\zeta^{0},\textbf{d}, \textbf{t})}+ \phi_{(\zeta^{0},\textbf{d}, \textbf{t}) }^\varepsilon\right)=J_{\varepsilon}\left( PU_{\delta_{1},\xi_{1}}-PU_{\delta_{2},\xi_{2}}+ \phi_{(\zeta^{0},\textbf{d}, \textbf{t}) }^\varepsilon\right) .$$ 
	
	\bigskip
	
	We will compute an asymptotic expansion of $\mathcal{I}^{2}_\varepsilon$.
	\begin{proposition}\label{exp2}
		The asymptotic expansion
		\[
		\mathcal{I}^{2}_\varepsilon(\zeta^{0},\textbf{d}, \textbf{t})=\mathcal{I}^{2}_\varepsilon(\textbf{d}, \textbf{t})=a(\zeta^{0})\omega_{1} +\varepsilon F^{2}(\zeta^{0},\textbf{d}, \textbf{t})+o(\varepsilon)
		\]
		holds true $C^{1}$-uniformly on compact subsets of $\Sigma_{2}$. Here the
		function $F^{2}:\Sigma_{2}\rightarrow \mathbb{R}$ is given by
		\begin{equation*}
			\begin{aligned}
				F^{2}(\zeta^{0},\textbf{d}, \textbf{t})&:=\nabla a(\zeta^{0})\cdot\nu(\zeta^{0})(t_{1}+t_{2})\omega_{2}-\left(\log (d_{1})+\log (d_{2}) \right)\omega_{3}\\
				+&a(\zeta^{0})\left( 2(d_{1}d_{2})^{\frac{N-4}{2}}\left(\frac{1}{|t_{1}-t_{2}|^{N-4}}-\frac{1}{|t_{1}+t_{2}|^{N-4}} \right)  + \left( \frac{d_{1}}{2t_{1}}\right) ^{N-4}+\left( \frac{d_{2}}{2t_{2}}\right) ^{N-4}\right) \omega_{4} 
			\end{aligned}
		\end{equation*}
		here 
		\begin{equation*}
			\omega_{1}=\frac{\gamma_{1}(p-2)}{p}+\varepsilon \left( -\frac{2\gamma_{1}}{p^{2}}+\frac{2\gamma_{3}}{p}-\log (\varepsilon)(N-3)\frac{\gamma_{1}}{p}\right)
		\end{equation*}
		and $\omega_{2}, \omega_{3}$ and $\omega_{4}$ are positive constants given by:
		\begin{equation*}
			\omega_{2}:=\frac{\gamma_{1}(p-2)}{2p}, \  \omega_{3}:=\frac{N-4}{2}\frac{\gamma_{1}}{p}  \mbox{ and } \ \omega_{4}:=\frac{\gamma_{2}}{2}
		\end{equation*}
	\end{proposition}
	\begin{proof}	
		
		\bigskip		
		We have that
		\begin{equation*}
			\begin{aligned}
				J_{\varepsilon}(	V_{(\zeta^{0},\textbf{d}, \textbf{t})}):&=\frac{1}{2} \int_{\Omega}a(x)\Delta\left(PU_{\delta_{1},\xi_{1}}-PU_{\delta_{2},\xi_{2}}\right)^{2}\,dx-\frac{1}{p-\varepsilon} \int_{\Omega}a(x)\left|PU_{\delta_{1},\xi_{1}}-PU_{\delta_{2},\xi_{2}}\right|^{p-\varepsilon}\,dx.
			\end{aligned}
		\end{equation*}
		We write:
		\begin{equation}\label{L3}
			\begin{aligned}
				& \frac{1}{2}\int_{\Omega}a(x)|\Delta(PU_{\delta_{1},\xi_{1}}-PU_{\delta_{2},\xi_{2}})|^{2}\,dx\\
				&=\frac{1}{2}\sum_{i=1}^{2}\int_{\Omega}a(x)\Delta(PU_{\delta_{i},\xi_{i}})\Delta(PU_{\delta_{i},\xi_{i}}) \,dx-\int_{\Omega}a(x)\Delta(PU_{\delta_{1},\xi_{1}})\Delta(PU_{\delta_{2},\xi_{2}})\,dx.
			\end{aligned}
		\end{equation}
		A direct calculation, for any $1\leq i\leq j \leq 2$, shows that:
		\begin{equation*}
			\begin{aligned}
				\int_{\Omega}a(x)\Delta(PU_{\delta_{i},\xi_{i}})\Delta(PU_{\delta_{j},\xi_{j}})\,dx
				&=\int_{\Omega}\Delta a(x)\Delta(PU_{\delta_{i},\xi_{i}})PU_{\delta_{j},\xi_{j}}\,dx+2\int_{\Omega}\nabla a(x)\nabla\left( \Delta(PU_{\delta_{i},\xi_{i}})\right)PU_{\delta_{j},\xi_{j}}\,dx\\
				&+\int_{\Omega}a(x)\Delta^{2}(PU_{\delta_{i},\xi_{i}})PU_{\delta_{j},\xi_{j}}\,dx\\
				&=\int_{\Omega}a(x)\Delta^{2}(PU_{\delta_{i},\xi_{i}})PU_{\delta_{j},\xi_{j}}\,dx+o(\varepsilon).\\
			\end{aligned}
		\end{equation*}
		The last equality is justified by Proposition \ref{est1}, which shows us that:
		\begin{equation*}
			\int_{\Omega}\Delta a(x)\Delta(PU_{\delta_{i},\xi_{i}})PU_{\delta_{j},\xi_{j}}\,dx=o(\varepsilon) \quad\mbox{ and }	\quad \int_{\Omega}\nabla a(x)\nabla\left( \Delta(PU_{\delta_{i},\xi_{i}})\right)PU_{\delta_{j},\xi_{j}}\,dx=o(\varepsilon),
		\end{equation*}
		for any $1\leq i\leq j \leq 2$.

		Thus, a straightforward calculation, applying estimate \eqref{isj} in Proposition \ref{propij}, shows that:
		\begin{equation}\label{L1}
			\begin{aligned}
				&\int_{\Omega}a(x)\Delta(PU_{\delta_{1},\xi_{1}})\Delta(PU_{\delta_{2},\xi_{2}})\,dx
				=\int_{\Omega}a(x)\Delta^{2}(PU_{\delta_{1},\xi_{1}})PU_{\delta_{2},\xi_{2}}\,dx+o(\varepsilon)\\
				&=\int_{B_{\eta}(\xi_{1})}a(x)\Delta^{2}(PU_{\delta_{1},\xi_{1}})PU_{\delta_{2},\xi_{2}}\,dx+\int_{\Omega\setminus B_{\eta}(\xi_{1})}a(x)\Delta^{2}(PU_{\delta_{1},\xi_{1}})PU_{\delta_{2},\xi_{2}}\,dx+o(\varepsilon)\\\\
				=& \varepsilon a(\zeta^{0})(d_{1}d_{2})^{\frac{N-4}{2}}\left(\frac{1}{|t_{1}-t_{2}|^{N-4}}-\frac{1}{|t_{1}+t_{2}|^{N-4}} \right) \gamma_{2}+o(\varepsilon)
			\end{aligned}
		\end{equation}
		
		\medskip
		
		On the other hand,  by Proposition \ref{tmi}, for $i=1,2$ we obtain:
		\begin{equation}\label{L2}
			\begin{aligned}
				&\int_{\Omega}a(x)\Delta^{2}(PU_{\delta_{i},\xi_{i}})PU_{\delta_{i},\xi_{i}}\,dx\\
				=&	\int_{B_{\eta}(\xi_{i})}a(x)U_{\delta_{i},\xi_{i}}^{\frac{N+4}{N-4}}PU_{\delta_{i},\xi_{i}}\,dx+\int_{\Omega\setminus B_{\eta}(\xi_{i})}a(x)U_{\delta_{i},\xi_{i}}^{\frac{N+4}{N-4}}PU_{\delta_{i},\xi_{i}}\,dx\\
				=&	\int_{B_{\eta}(\xi_{i})}a(x)U_{\delta_{i},\xi_{i}}^{\frac{2N}{N-4}}\,dx+\int_{B_{\eta}(\xi_{i})}a(x)U_{\delta_{i},\xi_{i}}^{\frac{N+4}{N-4}}(PU_{\delta_{i},\xi_{i}}-U_{\delta_{i},\xi_{i}})\,dx+o(\varepsilon)\\
				=&a(\zeta^{0})\gamma_{1}+\varepsilon \nabla a(\zeta^{0})\cdot\nu(\zeta^{0})t_{i}\gamma_{1}
				-\varepsilon a(\zeta^{0})\left( \frac{d_{i}}{2t_{i}}\right) ^{N-4}\gamma_{2}+o(\varepsilon)
			\end{aligned}
		\end{equation}
		
		\medskip
		
		Therefore, from \eqref{L3}, \eqref{L1} and \eqref{L2}, we obtain
		\begin{equation*}
			\begin{aligned}
				&\frac{1}{2} \int_{\Omega}a(x)|\Delta(PU_{\delta_{1},\xi_{1}}-PU_{\delta_{2},\xi_{2}})|^{2}\,dx\\
				&=a(\zeta^{0})\gamma_{1}+\varepsilon \nabla a(\zeta^{0})\cdot\nu(\zeta^{0})(t_{1}+t_{2})\frac{\gamma_{1}}{2}
				-\varepsilon a(\zeta^{0})\left( \left(\frac{d_{1}}{2t_{1}} \right) ^{N-4}+ \left( \frac{d_{2}}{2t_{2}}\right)^{N-4} \right) \frac{\gamma_{2}}{2}\\
				&- \varepsilon a(\zeta^{0})(d_{1}d_{2})^{\frac{N-4}{2}}\left(\frac{1}{|t_{1}-t_{2}|^{N-4}}-\frac{1}{|t_{1}+t_{2}|^{N-4}} \right) \gamma_{2}+o(\varepsilon)
			\end{aligned}
		\end{equation*}
		The term 	
		\begin{equation*}
			\frac{1}{p-\varepsilon} \int_{\Omega}a(x)\left|PU_{\delta_{1},\xi_{1}}-PU_{\delta_{2},\xi_{2}}\right|^{p-\varepsilon}\,dx,
		\end{equation*}
		is estimated in Lemma \ref{nonlinearij}. This completes the proof.
	\end{proof}

	\bigskip

	\begin{proof}{of Theorem \ref{thm2}}

		It is straightforward to verify that $F^{2}(\zeta^{0},\textbf{d}, \textbf{t})$  has a minimum point that remains stable under $C^{0}$-perturbations. Therefore, by Proposition \ref{exp2}, we conclude that if $\varepsilon$ is sufficiently small, the function $\mathcal{I}^{2}_\varepsilon(\zeta^{0},\textbf{d}, \textbf{t})$ possesses a critical point. The result then follows from Lemma \ref{reduc}.

	\end{proof}	

\bigskip
    
	\section{Boundary estimates of the Green function.}\label{sect:Green-function}
    This section focuses on the detailed analysis of the Green’s function and its regular part for the biharmonic operator under Navier boundary conditions. We derive precise estimates for these functions, particularly near the boundary of the domain. 

\medskip
    
	Let  $G(x, y)$ be  the  Green's function for the bi-Laplacian operator on $\Omega$ with Navier boundary condition and by $H(x,y)$ its regular part, i.e. $G(x,y)$ satisfies the problem:
	\begin{equation}
		\left\{
		\begin{array}
			[c]{ll}%
			\Delta_{y}^{2} G(x,y)=\gamma_{N}\delta_{x} & \text{in }\Omega,\\
			G(x,y)=\Delta_{y} G(x,y)=0&  \text{on }\partial\Omega
		\end{array}
		\right.  %
	\end{equation}
	where $\gamma_{N}:=(N-4)(N-2)meas(\mathbb{S}^{N-1})$ and $H(x,y)$  is defined by 
	\begin{equation*}
		H(x,y)=\frac{1}{|x-y|^{N-4}}-G(x,y)
	\end{equation*}
	and  verifies
	\begin{equation}\label{defH}
		\left\{
		\begin{array}
			[c]{ll}%
			\Delta_{y}^{2} H(x,y)=0 & \text{in }\Omega,\\
			H(x,y)=\frac{1}{|x-y|^{N-4}}&  \text{on }\partial\Omega\\
			\Delta_{y} H(x,y)=-2(N-4)\frac{1}{|x-y|^{N-2}}&  \text{on }\partial\Omega
		\end{array}
		\right.  %
	\end{equation}
	
	\bigskip

	We will begin by describing how $G(x,y)$ relates to the  $G_{\Delta}(x,y)$,   the Green's function for the Laplacian operator on  $\Omega$ with Dirichlet boundary condition.

	Let  $H_{\Delta}(x,y)$ be the regular part of $G_{\Delta}(x,y)$. Then
	\begin{equation*}
		H_{\Delta}(x,y)=\frac{1}{|x-y|^{N-2}}-G_{\Delta}(x,y).
	\end{equation*}
	From the equation \eqref{defH}, it is immediately evident that:
	$$H_{\Delta}(x,y):=\frac{-1}{2(N-4)}	\Delta_{y}H(x,y).$$  
	On the other hand, if  $u$ is a solution to the problem $\Delta^{2}u=f$ in $\Omega$ with $u=\Delta u=0$ in $\partial\Omega$ then:
	\begin{eqnarray*}
		\begin{aligned}
			u(x)&=\int_{\Omega}G_{\Delta}(x,z)(-\Delta u(z))\,dz\\
			&=\int_{\Omega}G_{\Delta}(x,z)\left(\int_{\Omega}G_{\Delta}(z,y) f(y)\,dy\right)\,dz\\
			&=\int_{\Omega}\left( \int_{\Omega}G_{\Delta}(x,z)G_{\Delta}(z,y)\, dz\right)  f(y)\,dy.\\
		\end{aligned}
	\end{eqnarray*}
	The preceding expression shows that
	\begin{equation*}
		G(x,y)= \int_{\Omega}G_{\Delta}(x,z)G_{\Delta}(z,y) \,dz.
	\end{equation*}
	Moreover, since $G_{\Delta}(x,y)$ is symmetric, i.e. $G_{\Delta}(x,y)=G_{\Delta}(y,x)$,it follows that
	\begin{equation*}
		G(x,y)= \int_{\Omega}G_{\Delta}(x,z)G_{\Delta}(z,y) \,dz= \int_{\Omega}G_{\Delta}(y,z)G_{\Delta}(z,x) \,dz=G(y,x).
	\end{equation*}
	Thus, $G(x,y)$ and $H(x,y)$ are symmetric functions. Furthermore, it holds true that:
	\begin{equation}\label{symH}
		\Delta_{y}H(x,y)=\Delta_{x}H(x,y)
	\end{equation}

	\bigskip
	Let $r>0$, and consider the set
	\begin{equation*}
		\Omega_{r}:=\{\xi\in\Omega: \,{\rm dist}(\xi,\partial\Omega)\leq r\}.
	\end{equation*}
	We will fix  $r$ small enough such that for every $\xi\in\Omega_{2r}$ there exists a unique $p(\xi)\in\partial \Omega$ with  ${\rm dist}(\xi,\partial\Omega)=|\xi-p(\xi)|$. We write $d_{\xi}:=dist(\xi,\partial\Omega)$,  and $\nu(\xi)$  will denote the inward normal to $\partial\Omega$ at $\xi$. For every $\xi\in\Omega_{2r}$ we define
	\begin{equation*}
		\widetilde{\xi}:=p(\xi)-d_{\xi}\nu(\xi)
	\end{equation*}
	the reflection of $\xi$ on $\partial\Omega$. 
	
	Note that, since the domain $\Omega$ is smooth, there exists a positive constant $C$ such that:
	\begin{equation}\label{font}
		\frac{|\widetilde{\xi}-y|}{|\xi-y|}\geq C \mbox{ for all }\xi\in \Omega_{r}, y\in \Omega
	\end{equation}
	
	\bigskip
	The next lemma give us an accurate estimate of $H(x,y)$ when the point $x$ is close to the boundary. 
	\begin{lemma}\label{LAH}
		There exists a positive constant $C>0$ such that
		\begin{equation}\label{esth}
			\begin{aligned}
				|H(x,y)-\frac{1}{|\widetilde{x}-y|^{N-4}}|&\leq \frac{C d_{x}}{|\widetilde{x}-y|^{N-4}},
			\end{aligned}
		\end{equation}
		\begin{equation}\label{estlaphy}
			\begin{aligned}
				|\Delta_{y}H(x,y)+\frac{2(N-4)}{|\widetilde{x}-y|^{N-2}}|&\leq \frac{C d_{x}}{|\widetilde{x}-y|^{N-2}},
			\end{aligned}
		\end{equation}
		and
		\begin{equation}\label{gradlap}
			|\nabla_{x}\left(\Delta_{y}H(x,y)+\frac{2(N-4)}{|\widetilde{x}-y|^{N-2}}\right) |\leq \frac{C}{|\widetilde{x}-y|^{N-2}}	
		\end{equation}
		for every $x\in \Omega_{r}$ and $y\in\Omega$.
	\end{lemma}
	\bigskip
	
	\begin{proof}
		We introduce the notation:
		\begin{equation*}
			\psi(x,y)=H(x,y)-\frac{1}{|\widetilde{x}-y|^{N-4}} \ \ \ x\in \Omega_{r}, y\in\Omega
		\end{equation*}
		Note that for each fixed $x\in \Omega_{r}$, the function 
		$\psi(x,y)$ satisfies the following problem:
		\begin{equation}
			\left\{
			\begin{array}
				[c]{ll}%
				\Delta_{y}^{2} \psi(x,y)=0 & \text{in }\Omega,\\
				\psi(x,y)=\frac{1}{|x-y|^{N-4}}-\frac{1}{|\widetilde{x}-y|^{N-4}}&  \text{on }\partial\Omega\\
				\Delta_{y} \psi(x,y)=-2(N-4)\left( \frac{1}{|x-y|^{N-2}}+\frac{1}{|\widetilde{x}-y|^{N-2}}\right) &  \text{on }\partial\Omega
			\end{array}
			\right.  %
		\end{equation}
		Since, 
		\begin{equation*}
			\sup_{y\in\partial\Omega}|\psi(x,y)|= \sup_{y\in\partial\Omega} | \left( \frac{1}{|x-y|^{N-4}}-\frac{1}{|\widetilde{x}-y|^{N-4}} \right)| \leq \frac{C d_{x}}{|\widetilde{x}-y|^{N-4}},
		\end{equation*}
		It follows from the Maximum Principle that 
		\begin{equation*}
			|H(x,y)-\frac{1}{|\widetilde{x}-y|^{N-4}}|\leq \frac{C d_{x}}{|\widetilde{x}-y|^{N-4}},  \ \  x\in \Omega_{r},y\in\Omega
		\end{equation*}
		This completes the proof of inequality \eqref{esth}. 
		\bigskip
		
		Next, we prove equation \eqref{estlaphy}. For each fixed $x\in \Omega_{r}$, the function
		\begin{equation*}
			\Delta_{y}\psi(x,y)=\Delta_{y}H(x,y)+\frac{2(N-4)}{|\widetilde{x}-y|^{N-2}}
		\end{equation*}
		satisfies the problem:
		\begin{equation}
			\left\{
			\begin{array}
				[c]{ll}%
				\Delta_{y} (\Delta_{y}\psi(x,y))=0 & \text{in }\Omega,\\
				\Delta_{y} \psi(x,y)=-2(N-4)\left( \frac{1}{|x-y|^{N-2}}-\frac{1}{|\widetilde{x}-y|^{N-2}}\right) &  \text{on }\partial\Omega
			\end{array}
			\right.  %
		\end{equation}
		Applying the maximum principle, we obtain:
		\begin{equation*}
			|\Delta_{y}H(x,y)+\frac{2(n-4)}{|\widetilde{x}-y|^{N-2}}|\leq \frac{C d_{x}}{|\widetilde{x}-y|^{N-2}}
		\end{equation*}
		This completes the proof of equation \eqref{estlaphy}.
		
		\bigskip
		
		To verify the validity of inequality  \eqref{gradlap}, we first prove that:
		\begin{equation}\label{ec2}
			|\Delta_{x}\Delta_{y}\psi(x,y)|=	|\Delta_{x}\left(\Delta_{y}H(x,y)+\frac{2(N-4)}{|\widetilde{x}-y|^{N-2}}\right) |\leq \frac{C}{d_{x}|\widetilde{x}-y|^{N-2}}
		\end{equation}
		for every $x\in \Omega_{r}$ and $y\in\Omega$. This inequality follows from the fact that
		$$\Delta_{x}\Delta_{y}H(x,y)=\Delta_{x}H_{\Delta}(x,y)=0$$
		and that, a direct calculation, which can be found in \cite[Lemma 2.2]{bartsch2010n}, shows that:
		\begin{equation*}
			|\Delta_{x}\left(\frac{2(N-4)}{|\widetilde{x}-y|^{N-2}}\right) |\leq \frac{C}{d_{x}|\widetilde{x}-y|^{N-2}} \ \ \ x\in\Omega_{r}, y\in\Omega
		\end{equation*}

		\bigskip
		
		We now proceed with the proof of the equation \eqref{gradlap}. We will rely on  from \cite[Theorem 3.9]{gilbarg1977elliptic}.  Fix $x\in \Omega_{r}$, $y\in\Omega$, and $r:=2\sqrt{N}$.  Consider the set
		\[
		B:=\{\xi\in \mathbb{R}^{N} \mid |x-\xi|_\infty \leq d_{x}/r\}.
		\]
		Note that if $\xi \in B$ then 
		\begin{equation*}
			|\xi-x|<d_x/2
		\end{equation*}
		which implies
		\begin{equation}\label{ec3}
			d_x/2\leq d_\xi\leq 3d_x/2.
		\end{equation}
		Then, for $i=1,2,\dots,N$, using inequalities \eqref{estlaphy}, \eqref{ec2} and \eqref{ec3},  we have:
		\begin{equation*}
			\begin{aligned}
				|\partial_{x_i}\Delta_{y}\psi(x,y)| &\leq \frac{rN}{d_x}\sup_{\xi \in \partial B}|\Delta_{y}\psi(\xi,y)|+\frac{d_x}{2r}\sup_{\xi\in B}|\Delta_\xi \Delta_{y}\psi(\xi,y)|\\
				&\leq C\left( \sup_{\xi \in \partial B}\frac{d_\xi}{d_x}\frac{1}{|\widetilde{\xi}-y|^{N-2}}+\sup_{\xi\in B}\frac{d_x}{d_\xi}\frac{1}{|\widetilde{\xi}-y|^{N-2}}\right)\\
				&\leq C\sup_{\xi\in B}\frac{1}{|\widetilde{\xi}-y|^{N-2}}\\
				&\leq \frac{C}{|\widetilde{x}-y|^{N-2}}.
			\end{aligned}
		\end{equation*}
		This completes the proof of inequality \eqref{gradlap}.
	\end{proof}
	
	\bigskip

	One consequence of the previous lemma is that, given the symmetry described by \eqref{symH}: 
	\begin{equation*}
		\Delta_{y}H(x,y)=\Delta_{x}H(x,y)
	\end{equation*}
	the inequalities \eqref{estlaphy} and \eqref{gradlap} can be rewritten as:
	\begin{equation}\label{estlaphx}
		\begin{aligned}
			\left|\Delta_{x}H(x,y)+\frac{2(N-4)}{|\widetilde{x}-y|^{N-2}}\right|&\leq \frac{C d_{x}}{|\widetilde{x}-y|^{N-2}}
		\end{aligned}
	\end{equation}
	and
	\begin{equation}\label{gradlap2}
		\begin{aligned}
			\left|\nabla_{x}\left(\Delta_{x}H(x,y)+\frac{2(N-4)}{|\widetilde{x}-y|^{N-2}}\right) \right|&\leq \frac{C}{|\widetilde{x}-y|^{N-2}}	\end{aligned}
	\end{equation}
	for all $x\in \Omega_{r}$ and $y\in\Omega$. Another consequence of Lemma \ref{LAH} is the following:
	
	\vspace{5mm}
	
	\begin{corollary}\label{coro1}
		There exists a positive constant $C>0$ such that
		\begin{equation}\label{estlaphx1}
			\begin{aligned}
				\left|\Delta_{y}H(x,y)\right|=	\left|\Delta_{x}H(x,y)\right|&\leq \frac{C }{|x-y|^{N-2}}
			\end{aligned}
		\end{equation}
		and
		\begin{equation}\label{gradlap21}
			\begin{aligned}
				\left|\nabla_{x}\Delta_{x}H(x,y)\right|&\leq \frac{C}{|x-y|^{N-1}}	\end{aligned}
		\end{equation}
		for every $x,y\in\Omega$.
	\end{corollary}
	\begin{proof}
		
		\bigskip
		
		Since $H(x,y)$ satisfies the problem \eqref{defH}, there exists a constant $C>0$ such that
		\begin{equation}\label{acoth}
			|\Delta_{x}H(x,y)|\leq C
		\end{equation} 
		and 
		\begin{equation}\label{acothg}
			|\nabla_{x}\Delta_{x}H(x,y)|\leq C
		\end{equation} 
		for every $x\in\Omega\setminus \Omega_{r}$ and every $y\in\Omega$.
		
		\bigskip
		
		We will now prove inequality \eqref{estlaphx1}. Note that, by inequality \eqref{font}, it follows that:
		\begin{equation}\label{uio1}
			\frac{2(N-4)}{|\widetilde{x}-y|^{N-2}}\leq \frac{C}{|x-y|^{N-2}} \mbox{ for all }x\in \Omega_{r}, y\in \Omega 
		\end{equation}for some positive constant $C$.
		Thus, the inequality \eqref{estlaphx1} follows from \eqref{acoth}, \eqref{estlaphx}, and \eqref{uio1}.

		\bigskip
		
		We will now focus on proving the inequality \eqref{acothg}. Direct computations, similar to those in the proof of  \cite{bartsch2010n}, show that
		
		\begin{equation*} \left| \nabla_{x} \left( \frac{2(N-4)}{|\widetilde{x} - y|^{N-2}} \right) \right| \leq \frac{C}{|x - y|^{N-1}}. 
		\end{equation*}
		Therefore, using \eqref{gradlap2} and  \eqref{uio1}, we obtain
		\begin{equation}
			\begin{aligned}
				|\nabla_{x}\left(\Delta_{x}H(x,y)\right) |&\leq \frac{C}{|\widetilde{x}-y|^{N-2}}+|\nabla_{x}\left(\frac{2(N-4)}{|\widetilde{x}-y|^{N-2}}\right)|\\
				&\leq \frac{C}{|x-y|^{N-2}}+\frac{C}{|x-y|^{N-1}}\\
				&\leq \frac{C}{|x-y|^{N-1}}	\end{aligned}
		\end{equation}
		for every $x\in \Omega_{r}$ and $y\in\Omega$. This completes the proof of the corollary.
	\end{proof}
	
	\bigskip
	\begin{remark}
		Given the Corollary \ref{coro1}, we have the following estimates for the Green's function $G_{\Delta}(x,y)$  for the Laplacian operator on  $\Omega$ with Dirichlet boundary condition:
		\begin{equation}\label{G1}
			\begin{aligned}
				|G_{\Delta}(x,y)|&=|\frac{1}{|x-y|^{N-2}}-H_{\Delta}(x,y)|\\
				&=|\frac{1}{|x-y|^{N-2}}-\Delta_{x}H(x,y)|\leq \frac{C }{|x-y|^{N-2}}
			\end{aligned}
		\end{equation}
		and 
		\begin{equation}\label{G2}
			\begin{aligned}
				|\nabla_{x}G_{\Delta}(x,y)|&=|\nabla_{x}(\frac{1}{|x-y|^{N-2}}-H_{\Delta}(x,y))|\\
				&=|\nabla_{x}\left(  \frac{1}{|x-y|^{N-2}}-\Delta_{x}H(x,y)\right)|\leq \frac{C }{|x-y|^{N-1}}
			\end{aligned}
		\end{equation}
		for every $x,y\in\Omega$.
	\end{remark}

	\bigskip\bigskip
	\section{Approximation of Solutions and Bubble Projections}\label{sect:Aprox-solution}
In this section, we establish the technical estimates required for analyzing the projection of bubble functions. These approximations allow us to quantify the discrepancy between the bubble solutions and their projections onto the functional space. This analysis provides us with the asymptotic behavior of the solutions.
    
    %Aproximacion de las soluciones y cuentas necesarias sobre la proyecccion de las burbujas}

%In this section, we establish the technical estimates that we have ....
    
\medskip

	We begin by exploring the relationship between the bubble  $U_{\delta,\xi}(x)$ and its projection $PU_{\delta,\xi}(x)$ when the point $\xi$ is close to the boundary $\partial\Omega$. 
	\begin{proposition}\label{Propaprox}
		Let $\delta\in (0,1]$ and $\xi\in \Omega_{r}$. Let $\widetilde{\xi}$ be the reflection of $\xi$ with respect to the boundary $\partial\Omega$. Then, there exits a constant $c>0$ such that:
		\begin{equation}\label{equa1}
			0\leq PU_{\delta,\xi}(x)\leq U_{\delta,\xi}(x) 
		\end{equation}	
		and 
		\begin{equation}\label{equa2}
			0\leq U_{\delta,\xi}(x)-PU_{\delta,\xi}(x)\leq \alpha_{N}\delta^{\frac{N-4}{2}}H(x,\xi)\leq c\frac{\delta^{\frac{N-4}{2}}}{|x-\widetilde{\xi}|^{N-4}}
		\end{equation}		
		for every $x\in\Omega$. Furthermore,
		\begin{equation*}
			R_{\delta,\xi}(x):= PU_{\delta,\xi}(x)-U_{\delta,\xi}(x)+\alpha_{N}\delta^{\frac{N-4}{2}}H(x,\xi)
		\end{equation*}
		where the remainder term $R_{\delta,\xi}$ satisfies the following estimate:
		\begin{equation}\label{res1}
			|R_{\delta,\xi}|_{\Omega,\infty}	=O\left( \frac{\delta^{\frac{N}{2}}}{dist(\xi,\partial\Omega)^{N-2}}\right).
		\end{equation}	
		
	\end{proposition}
	
	\begin{proof}
		
		\bigskip	
		The inequalities \eqref{equa1} and \eqref{equa2} are derived directly from the maximum principle.  We will focus on proving the inequality \eqref{res1}. 
		Note that the function  $R_{\delta,\xi}$  satisfies the following boundary value problem: 
		\begin{equation}
			\left\{
			\begin{array}
				[c]{ll}%
				\Delta_{x}^{2} R_{\delta,\xi}(x)=0 & \text{in }\Omega,\\
				\Delta_{x}R_{\delta,\xi}(x)=-	\Delta_{x}U_{\delta,\xi}(x)-\alpha_{N}2(N-4)\delta^{\frac{N-4}{2}}\frac{1}{|x-\xi|^{N-2}}&  \text{on }\partial\Omega\\
				R_{\delta,\xi}(x)=-U_{\delta,\xi}(x)+\alpha_{N}\delta^{\frac{N-4}{2}}\frac{1}{|x-\xi|^{N-4}}&  \text{on }\partial\Omega
			\end{array}
			\right.  %
		\end{equation}	
		Thus, by applying the maximum principle, we obtain:
		\begin{equation*}
			\begin{aligned}
				|R_{\delta,\xi}|_{\Omega,\infty}&\leq \alpha_{N}\delta^{\frac{N-4}{2}}\max_{x\in\partial\Omega}\left|\frac{1}{(\delta^{2}+|x-\xi|^{2})^{\frac{N-4}{2}}} -\frac{1}{|x-\xi|^{N-4}}\right| _{\Omega,\infty}\\
				&= O\left(\delta^{\frac{N-4}{2}} \max_{x\in\partial\Omega}\frac{\delta^{2}}{|x-\xi|^{N-2}}\right)= O\left( \frac{\delta^{\frac{N}{2}}}{dist(\xi,\partial\Omega)^{N-2}}\right) \\
			\end{aligned}
		\end{equation*}
This completes the desired estimate.

	\end{proof}
	
	This previous result provides a precise bound on the difference between the bubble  $U_{\delta,\xi}(x)$ and its projection $PU_{\delta,\xi}(x)$, which plays a crucial role in our analysis of the asymptotic behavior. We use this result to prove:
	
	\begin{remark}\label{Rmk2}
		Let $\xi_{i}, \xi_{j}\in \Omega$ and such that
		\begin{equation*}
			\xi_{i}:=\xi_{i}^{0}+ \tau_{i}\nu(\xi_{i}^{0}) \mbox{ and  } 	\xi_{j}:=\xi_{j}^{0}+ \tau_{j}\nu(\xi_{j}^{0})
		\end{equation*}
		where 
		\begin{equation*}
			\tau_{i}=\varepsilon t_{i}  \mbox{and  } \tau_{i}=\varepsilon t_{j} \mbox{ with }  t_{i}<t_{j}.
		\end{equation*}
		Set 
		\begin{equation*}
			\eta:=\min\{d(\xi_{1}, \partial\Omega), d(\xi_{2}, \partial\Omega), \frac{|\xi_{1}-\xi_{2}|}{2} \}=\min\{t_{1}\varepsilon,t_{2}\varepsilon,\frac{\varepsilon|t_{1}-t_{2}|}{2} \}
		\end{equation*}
		and, as before:
		\begin{equation*}
			R_{\delta_{j},\xi_{j}}(x) =	PU_{\delta_{j},\xi_{j}}(x)	-U_{\delta_{j},\xi_{j}}(x)+\alpha_{N}\delta_{j}^{\frac{N-4}{2}}H(x,\xi_{j})
		\end{equation*}
		Then: 
		\begin{equation}\label{EstRes}
			\begin{aligned}
				\int_{B_{\eta}(\xi)}a(x)U_{\delta_{i},\xi_{i}}^{\frac{N+4}{N-4}}(x)R_{\delta_{j},\xi_{j}}(x) \,dx=O\left( \frac{\delta_{j}^{\frac{N}{2}}\delta_{i}^{\frac{N-4}{2}}}{\varepsilon^{N-2}}\right) 
			\end{aligned}
		\end{equation}
		and 
		\begin{equation}\label{EstH}
			\begin{aligned}
				\int_{B_{\eta}(\xi_{i})}U_{\delta_{i},\xi_{i}}^{\frac{N+4}{N-4}}(x)\left( \alpha_{N}\delta_{j}^{\frac{N-4}{2}}H(x,\xi_{j})\right)\,dx &=	\int_{B_{\eta}(\xi_{i})}U_{\delta_{i},\xi_{i}}^{\frac{N+4}{N-4}}(x)\frac{\alpha_{N}\delta_{j}^{\frac{N-4}{2}}}{|x-\widetilde{\xi}_{j}|^{N-4}}\,dx+O\left( \varepsilon\frac{(\delta_{i}\delta_{j})^{\frac{N-4}{2}}}{\varepsilon^{N-4}} \right)
			\end{aligned}
		\end{equation}
	\end{remark}
	
	\bigskip
	\begin{proof}
		To prove the equation \eqref{EstRes}, we directly used estimation \eqref{res1} to obtain:
		\begin{equation*}
			\begin{aligned}
				\int_{B_{\eta}(\xi)}a(x)U_{\delta_{i},\xi_{i}}^{\frac{N+4}{N-4}}(x)R_{\delta_{j},\xi_{j}}(x) \,dx &\leq C\left(\frac{\delta_{j}^{\frac{N}{2}}}{\varepsilon^{N-2}} \right)\int_{B_{\eta}(\xi)}\frac{\delta_{i}^{\frac{N+4}{2}}}{\left(\delta_{i}^{2}+|x-\xi_{i}|^{2} \right)^{\frac{N+4}{2}} }\,dx\\
				&=C\left(\frac{\delta_{j}^{\frac{N}{2}}}{\varepsilon^{N-2}} \right)\int_{B_{\frac{\eta}{\delta_{i}}}(0)}\frac{\delta_{i}^{N+\frac{N+4}{2}}}{\left(\delta_{i}^{2}+|\delta_{i}y|^{2} \right)^{\frac{N+4}{2}} }\,dx\\
				&=O\left( \frac{\delta_{j}^{\frac{N}{2}}\delta_{i}^{\frac{N-4}{2}}}{\varepsilon^{N-2}}\right) 
			\end{aligned}
		\end{equation*}
		
		i.e. 
		\begin{equation*}
			\begin{aligned}
				\int_{B_{\eta}(\xi_{i})}a(x)U_{\delta_{i},\xi_{i}}^{\frac{N+4}{N-4}}(x)R_{\delta_{j},\xi_{j}}(x) \,dx=O\left( \frac{\delta_{j}^{\frac{N}{2}}\delta_{i}^{\frac{N-4}{2}}}{\varepsilon^{N-2}}\right)  
			\end{aligned}
		\end{equation*}
		
		Now we will prove \eqref{EstH}. Using the estimation given in \eqref{esth}, we obtain:
		\begin{equation*}
			\begin{aligned}
				&\int_{B_{\eta}(\xi_{i})}U_{\delta_{i},\xi_{i}}^{\frac{N+4}{N-4}}(x)\left(\alpha_{N}\,\delta_{j}^{\frac{N-4}{2}}H(x,\widetilde{\xi}_{j})-\frac{\alpha_{N}\delta_{j}^{\frac{N-4}{2}}}{\left|x-\widetilde{\xi}_{j}\right|^{N-4}}\right)\,dx \\
				&= O\left(\delta_{j}^{\frac{N-4}{2}} \int_{B_{\eta}(\xi_{i})}U_{\delta_{i},\xi_{i}}^{\frac{N+4}{N-4}}(x)\frac{\tau_{j}}{\left|x-\widetilde{\xi}_{j}\right|^{N-4}}\,dx\right) \\
				&=O\left( \varepsilon\,\delta_{j}^{\frac{N-4}{2}} \int_{B_{\frac{\eta}{\delta_{i}}}(0)}\frac{\delta_{i}^{N-\frac{N+4}{2}}}{\left(1+|y|^{2}\right)^{\frac{N+4}{2}} } \frac{1}{\left|\delta_{i}y+\xi_{i}-\widetilde{\xi}_{j}\right|^{N-4}}\,dy\right) \\
		&=O\left(\varepsilon\,\delta_{j}^{\frac{N-4}{2}} \dfrac{\delta_{i}^{\frac{N-4}{2}}}{\varepsilon^{N-4}}\int_{B_{\frac{\eta}{\delta_{i}}}(0)}\dfrac{1}{\left(1+|y|^{2}\right)^{\frac{N+4}{2}} } \frac{1}{\left|\frac{\delta_{i}}{\varepsilon}y+\frac{\xi_{i}-\widetilde{\xi}_{j}}{\varepsilon}\right|^{N-4}}\,dy\right) \\
				&=O\left( \varepsilon\,\dfrac{(\delta_{i}\,\delta_{j})^{\frac{N-4}{2}}}{\varepsilon^{N-4}}\int_{B_{\dfrac{\eta}{\delta_{i}}}(0)}\dfrac{1}{\left(1+|y|^{2}\right)^{\frac{N+4}{2}} } \,dy\right)=O\left( \varepsilon\,\dfrac{\left(\delta_{i}\,\delta_{j}\right)^{\frac{N-4}{2}}}{\varepsilon^{N-4}} \right)=o(\varepsilon),\\
			\end{aligned}
		\end{equation*}
		because, since $y\in B_{\frac{\eta}{\delta_{i}}}(0)$, then $\frac{\delta_{i}}{\varepsilon}|y|<\frac{\eta}{\varepsilon}$ and therefore:
		\begin{equation*}
			\left|\frac{\delta_{i}}{\varepsilon}y+\frac{\xi_{i}-\widetilde{\xi}_{j}}{\varepsilon}\right|\geq 	\left|\frac{\xi_{i}-\widetilde{\xi}_{j}}{\varepsilon}\right|-\left|\frac{\delta_{i}}{\varepsilon}y\right|= (t_{1}+t_{2})-\min\left\{t_{1}, t_{2},\frac{|t_{1}-t_{2}|}{2} \right\}\geq \min\left\{t_{1}, t_{2},\frac{|t_{1}-t_{2}|}{2} \right\}.
		\end{equation*}
	\end{proof}

	\bigskip
	
	In the rest of this work, we will use Remark \ref{Rmk2} to prove various estimates. We begin by proving the following asymptotic expansions:

	\begin{proposition}\label{tmi} Let $\xi_{i}\in\Omega$ and $\eta$ defined as in Remark \ref{Rmk2}, then:
		\begin{equation*}\label{tmi1}
			\int_{B_{\eta}(\xi_{i})}a(x)U_{\delta_{i},\xi_{i}}^{\frac{2N}{N-4}} (x)\,dx=a(\xi_{i}^{0})\gamma_{1}+\varepsilon t_{i}\nabla a(\xi_{i}^{0})\cdot\nu(\xi_{i}^{0})\gamma_{1}+o(\varepsilon^{1+\sigma}),
		\end{equation*}
		and
		\begin{equation*}\label{tmi2}
			\begin{aligned}
				\int_{B_{\eta}(\xi_{i})}a(x)\,U_{\delta_{i},\xi_{i}}^{\frac{N+4}{N-4}}(x)\left(PU_{\delta_{i},\xi_{i}}(x)-U_{\delta_{i},\xi_{i}}(x) \right)\,dx =-\varepsilon\left( \frac{d_{i}}{2t_{i}}\right) ^{N-4}a(\xi_{i}^{0})\,\gamma_{2}+o\left(\varepsilon^{1+\sigma}\right).
			\end{aligned}
		\end{equation*}
		Here:
		\begin{equation*}
			\gamma_{1}=	\alpha_{N}^{p}\int_{\mathbb{R}^{N}}\frac{1}{\left( 1+|y|^{2}\right)^{N} }\, dy \ \ \quad\mbox{ and }\quad	\ \ \gamma_{2}:=a_{N}^{p}\int_{\mathbb{R}^{N}}\frac{1}{\left(1+|y|^{2}\right)^{\frac{N+4}{2}} }\,dy.
		\end{equation*}
	\end{proposition}

\bigskip

	\begin{proof}
		
		We write: 
		\begin{eqnarray*}
			\begin{aligned}
				\int_{B_{\eta}(\xi_{i})}a(x)\,U_{\delta_{i},\xi_{i}}^{\frac{2N}{N-4}}(x)\,dx &=	\int_{B_{\eta}(\xi_{i})}(a(x)-a(\xi_{i}^{0}))U_{\delta_{i},\xi_{i}}^{\frac{2N}{N-4}}(x)\,dx  +	\int_{B_{\eta}(\xi_{i})}a(\xi_{i}^{0})U_{\delta_{i},\xi_{i}}^{\frac{2N}{N-4}}(x)\,dx . 
			\end{aligned}
		\end{eqnarray*}
		On one hand, we obtain:	
		\begin{equation}
			\begin{aligned}
				\int_{B_{\eta}(\xi_{i})}a(\xi_{i}^{0})U_{\delta_{i},\xi_{i}}^{\frac{2N}{N-4}}(x)\,dx &=\alpha_{N}^{p}a(\xi_{i}^{0})\int_{B_{\eta}(\xi_{i})}\frac{\delta_{i}^{N}}{\left( \delta_{i}^{2}+|x-\xi_{i}|^{2}\right)^{N} }\,dx\\
				&=\alpha_{N}^{p}a(\xi_{i}^{0})\int_{B_{\frac{\eta}{\delta_{i}}}(0)}\frac{1}{\left( 1+|y|^{2}\right)^{N} }\,dy\\
				&=a(\xi_{i}^{0})\left(\gamma_{1}+O\left( \left( \frac{\delta_{i}}{\varepsilon} \right)^{N} \right)\right)=a(\xi_{i}^{0})\gamma_{1}+o(\varepsilon).
			\end{aligned}
		\end{equation}
		On the other hand, since 
        
        \smallskip

		\begin{equation*}
a(\delta_{i}y+\xi_{i}^{0}+\tau_{i}\nu(\xi_{i}^{0}))=a(\xi_{i}^{0})+\nabla a(\xi_{i}^{0})\cdot(\tau_{i}\nu(\xi_{i}^{0})+\delta_{i} y)+R(y),
		\end{equation*}

        \medskip

	\noindent	where 	$|R(y)|=O\left( |\delta_{i}^{2}|y|^{2}+\delta_{i}\tau_{i}|y|+\tau_{i}^{2}|\right)$, then we have
		\begin{eqnarray*}
			\begin{aligned}
				\int_{B_{\eta}(\xi_{i})}(a(x)-a(\xi_{i}^{0}))U_{\delta_{i},\xi_{i}}^{\frac{2N}{N-4}}(x)\,dx&=\displaystyle\int_{B_{\eta}(\xi_{i})}\left(a(x)-a(\xi_{i}^{0})\right)U_{\delta_{i},\xi_{i}}^{\frac{2N}{N-4}}(x)\,dx\\
				&=\alpha_{N}^{p}\displaystyle\int_{B_{\frac{\eta}{\delta_{i}}}(0)}\frac{\left(a(\delta_{i}y+\xi_{i}^{0}+\tau_{i}\nu(\xi_{i}^{0}))-a(\xi_{i}^{0})\right)}{\left( 1+|y|^{2}\right)^{N} }\,dy\\
				&=\alpha_{N}^{p}\displaystyle\int_{B_{\frac{\eta}{\delta_{i}}}(0)}\frac{\nabla a(\xi_{i}^{0})\cdot(\tau_{i}\nu(\xi_{i}^{0})+\delta_{i} y)+R(y)}{\left( 1+|y|^{2}\right)^{N} }\,dy\\
				&=\varepsilon t_{i}\nabla a(\xi_{i}^{0})\cdot\nu(\xi_{i}^{0})\gamma_{1}+o(\varepsilon).
			\end{aligned}
		\end{eqnarray*}

\medskip
        
		Therefore:
		\begin{equation*}
			\int_{B_{\eta}(\xi_{i})}a(x)U_{\delta_{i},\xi_{i}}^{\frac{2N}{N-4}}(x)\,dx=a(\xi_{i}^{0})\gamma_{1}+\tau_{i}\nabla a(\xi_{i}^{0})\cdot\nu(\xi_{i}^{0})\gamma_{1}+o(\varepsilon).
		\end{equation*}
		
		\medskip
		
		Now, we continue with the expansion of the term 
		\begin{equation*}
			\int_{B_{\eta}(\xi_{i})}a(x)U_{\delta_{i},\xi_{i}}^{\frac{N+4}{N-4}}\left(PU_{\delta_{i},\xi_{i}}(x)-U_{\delta,\xi_{i}}(x) \right)\,dx
		\end{equation*}
		By equations  \eqref{EstRes} and \eqref{EstH}, we obtain: 
		\begin{equation*}
			\begin{aligned}
				&\int_{B_{\eta}(\xi_{i})}a(x)\,U_{\delta_{i},\xi_{i}}^{\frac{N+4}{N-4}}(x)\left(PU_{\delta_{i},\xi_{i}}(x)-U_{\delta_{i},\xi_{i}}(x) \right)\,dx \\
				=&	\int_{B_{\eta}(\xi_{i})}a(x)\,U_{\delta_{i},\xi_{i}}^{\frac{N+4}{N-4}}(x)\left( -\alpha_{N}\,\delta_{i}^{\frac{N-4}{2}}H(x,\xi_{i})+R_{\delta_{i},\xi_{i}}(x)\right)\,dx\\
				=&-\int_{B_{\eta}(\xi_{i})}a(x)\,U_{\delta_{i},\xi_{i}}^{\frac{N+4}{N-4}}(x)\frac{\alpha_{N}\,\delta_{i}^{\frac{N-4}{2}}}{|x-\widetilde{\xi}_{i}|^{N-4}}\,dx+o(\varepsilon)\\
				=&-\alpha_{N}^{p}\,\delta_{i}^{N-4}\int_{B_{\frac{\eta}{\delta_{i}}}(0)}a\left(\delta y+\xi_{i}\right)\frac{1}{\left(1+|y|^{2}\right)^{\frac{N+4}{2}} }  \left(\frac{1}{|\delta_{i}y+\xi_{i}-\widetilde{\xi_{i}}|^{N-4}}\right)\,dy+o(\varepsilon)\\
				=&-\alpha_{N}^{p}\, \delta_{i}^{N-4}\, a(\xi_{i}^{0})\left(\int_{B_{\frac{\eta}{\delta_{i}}}(0)}\frac{1}{\left(1+|y|^{2}\right)^{\frac{N+4}{2}} }  \left(\frac{1}{\left|\delta_{i}y+\xi_{i}^{0}+\tau_{i}\nu (\xi_{i}^{0})-(\xi_{i}^{0}-\tau_{i}\nu (\xi_{i}^{0}))\right|^{N-4}}\right)\,dy\right)\left(1+O(\varepsilon)\right)\\
				=&-\alpha_{N}^{p}\,\delta_{i}^{N-4}\,a(\xi_{i}^{0})\left(\int_{B_{\frac{\eta}{\delta_{i}}}(0)}\frac{1}{\left(1+|y|^{2}\right)^{\frac{N+4}{2}} }  \left(\frac{1}{\left|\delta_{i}y+2\tau_{i}\nu (\xi_{i}^{0})\right|^{N-4}}\right)\,dy\right)(1+O(\varepsilon))\\
				=&-\left( \frac{\delta_{i}}{2\tau_{i}}\right) ^{N-4}a(\xi_{i}^{0})\,\gamma_{2}\,(1+O(\varepsilon))=-\varepsilon\left( \frac{d_{i}}{2t_{i}}\right) ^{N-4}a(\xi_{i}^{0})\,\gamma_{2}+o(\varepsilon).
			\end{aligned}
		\end{equation*}	
		The last statement is true because:
		\begin{equation*}
			\left|\delta_{i}y+2\tau_{i}\eta (\xi_{i}^{0})\right|^{-(N-4)}=|2\tau_{i}\eta (\xi_{i}^{0})|^{-(N-4)}+O\left(\frac{\tau_{i}\delta_{i}|y|}{\tau_{i}^{N}}\right).
		\end{equation*}
		This concludes the proof of the lemma \ref{tmi}.
	\end{proof}

	\bigskip
	
	\begin{proposition}\label{propij}
		Let $\xi_{i},\xi_{j}\in\Omega$ and $\eta$ defined as in Remark \ref{Rmk2}.

		If $\xi_{i}^{0}\neq \xi_{j}^{0}$, then
		\begin{equation}\label{inj}
			\int_{B_{\eta}(\xi_{i})}a(x)\Delta^{2}(PU_{\delta_{i},\xi_{i}})PU_{\delta_{j},\xi_{j}}\,dx=o(\varepsilon)\ \ \
		\end{equation}
		If $\xi_{i}^{0}= \xi_{j}^{0}$, then
		\begin{equation}\label{isj}
			\int_{B_{\eta}(\xi_{i})}a(x)\Delta^{2}(PU_{\delta_{i},\xi_{i}})PU_{\delta_{j},\xi_{j}}\,dx= a(\xi_{i}^{0})\varepsilon(d_{i}d_{j})^{\frac{N-4}{2}}\left(\frac{1}{|t_{i}-t_{j}|^{N-4}}-\frac{1}{|t_{i}+t_{j}|^{N-4}} \right) \gamma_{2}+o(\varepsilon).
		\end{equation}
		
	\end{proposition}

	\bigskip
	
	\begin{proof}
		\medskip
		Let us first observe that
		\begin{equation*}
			\begin{aligned}
				&\alpha_{N}^{p}\int_{B_{\eta}(\xi)}a(x)\frac{\delta_{i}^{\frac{N+4}{2}}}{(\delta_{i}^{2}+|x-\xi_{i}|^{2})^{\frac{N+4}{2}}}  \left(\frac{\delta_{j}^{\frac{N-4}{2}}}{(\delta_{j}^{2}+|x-\xi_{j}|^{2})^{\frac{N-4}{2}}} \right)\,dx\\
				&=\alpha_{N}^{p}\int_{B_{\eta}(\xi)}a(x)\frac{\delta_{i}^{\frac{N+4}{2}}}{(\delta_{i}^{2}+|x-\xi_{i}|^{2})^{\frac{N+4}{2}}}  \left(\frac{\delta_{j}^{\frac{N-4}{2}}}{|x-\xi_{j}|^{N-4}} \right)\,dx\\
			\end{aligned}
		\end{equation*}
		This equation holds because
		\begin{equation*}
			\left(\delta_{j}^{2}+|x-\xi_{j}|^{2}\right)^{-\frac{N-4}{2}}=	|x-\xi_{j}|^{-(N-4)}-\frac{N-4}{2}\delta_{j}^{2}	\left(\theta_{x}\delta_{j}^{2}+|x-\xi_{j}|^{2}\right)^{-\frac{N-2}{2}}
		\end{equation*}	
		we have that
		\begin{equation*}
			\begin{aligned}
				&\int_{B_{\eta}(\xi_{i})}\dfrac{\delta_{i}^{\frac{N+4}{2}}}{\left(\delta_{i}^{2}+\left|x-\xi_{i}\right|^{2}\right)^{\frac{N+4}{2}}}  \frac{\delta_{j}^{\frac{N-4}{2}}\delta_{j}^{2}}{\left(\theta_{x}\delta_{j}^{2}+\left|x-\xi_{j}\right|^{2}\right)^{\frac{N-2}{2}}}\,dx \\
				&=(\delta_{i}\delta_{j})^{\frac{N-4}{2}}\delta_{j}^{2}\int_{B_{\frac{\eta}{\delta_{i}}}(0)}\dfrac{1}{\left(1+\left|y\right|^{2}\right)^{\frac{N+4}{2}}}  \frac{1}{\left(\theta_{x}\delta_{j}^{2}+\left|\delta_{i}y+\xi_{i}-\xi_{j}\right|^{2}\right)^{\frac{N-2}{2}}} \,dy\\
				&=\frac{(\delta_{i}\delta_{j})^{\frac{N-4}{2}}}{\varepsilon^{N-2}}\delta_{j}^{2}\int_{B_{\frac{\eta}{\delta_{i}}}(0)}\dfrac{1}{\left(1+\left|y\right|^{2}\right)^{\frac{N+4}{2}}}  \frac{1}{\left(\theta_{x}(\frac{\delta_{j}}{\varepsilon})^{2}+\left|\frac{\delta_{i}}{\varepsilon}y+\frac{\xi_{i}-\xi_{j}}{\varepsilon}\right|^{2}\right)^{\frac{N-2}{2}}} \,dy\\
				&\leq C\frac{(\delta_{i}\delta_{j})^{\frac{N-4}{2}}}{\varepsilon^{N-2}}\delta_{j}^{2}\int_{B_{\frac{\eta}{\delta_{i}}}(0)}\frac{1}{\left(1+\left|y\right|^{2}\right)^{\frac{N+4}{2}}}  \frac{1}{\left(\left|\frac{\delta_{i}}{\varepsilon}y+\frac{\xi_{i}-\xi_{j}}{\varepsilon}\right|^{2}\right)^{\frac{N-2}{2}}}\,dy \\
				&=O\left(\frac{(\delta_{i}\delta_{j})^{\frac{N-4}{2}}}{\varepsilon^{N-4}}\left(\frac{\delta_{j}}{\varepsilon} \right) ^{2}\int_{\mathbb{R}^{N}}\frac{1}{(1+\left|y\right|^{2})^{\frac{N+4}{2}}}\,dx \right)  \\
				&=o(\varepsilon),
			\end{aligned}
		\end{equation*}
		because, since
		\begin{equation*}
			\eta:=\min\left\{d(\xi_{i}, \partial\Omega), d(\xi_{j}, \partial\Omega), \frac{|\xi_{i}-\xi_{j}|}{2} \right\},
		\end{equation*}
		then, for  $y\in B_{\frac{\eta}{\delta_{i}}}(0)$ we have
		\begin{equation*}
			\left|\frac{\delta_{i}}{\varepsilon}y+\frac{|\xi_{i}-\xi_{j}|}{\varepsilon}\right|\geq \frac{|\xi_{i}-\xi_{j}|}{\varepsilon}-\left|\frac{\delta_{i}}{\varepsilon}y\right|>\frac{|\xi_{i}-\xi_{j}|}{\varepsilon}-\frac{\eta}{\varepsilon}>\frac{\eta}{\varepsilon}=O(1).
		\end{equation*}

		Hence, using equations  \eqref{EstRes} and \eqref{EstH}, we obtain: 
		\begin{equation}\label{estp1}
			\begin{aligned}
				&\int_{B_{\eta}(\xi_{i})}a(x)\Delta^{2}(PU_{\delta_{i},\xi_{i}})PU_{\delta_{j},\xi_{j}}\,dx\\
				=&\int_{B_{\eta}(\xi)}a(x)U_{\delta_{i},\xi_{i}}^{\frac{N+4}{N-4}}\left(U_{\delta_{j},\xi_{j}}(x)-\alpha_{N}\delta_{j}^{\frac{N-4}{2}}H(x,\xi_{j})+R_{\delta_{j},\xi_{j}}(x) \right)\,dx\\
				=&\alpha_{N}^{p}\int_{B_{\eta}(\xi)}a(x) \frac{\delta_{i}^{\frac{N+4}{2}}}{(\delta_{i}^{2}+|x-\xi_{i}|^{2})^{\frac{N+4}{2}}}  \left(\frac{\delta_{j}^{\frac{N-4}{2}}}{(\delta_{j}^{2}+|x-\xi_{j}|^{2})^{\frac{N-4}{2}}} -\delta_{j}^{\frac{N-4}{2}}H(x,\xi_{j})\right)\,dx+o\left(\varepsilon\right) \\
				=&\alpha_{N}^{p}\int_{B_{\eta}(\xi)}a(x) \frac{\delta_{i}^{\frac{N+4}{2}}}{(\delta_{i}^{2}+|x-\xi_{i}|^{2})^{\frac{N+4}{2}}}  \left(\frac{\delta_{j}^{\frac{N-4}{2}}}{(\delta_{j}^{2}+|x-\xi_{j}|^{2})^{\frac{N-4}{2}}} -\frac{\delta_{j}^{\frac{N-4}{2}}}{|\xi_{i}-\widetilde{\xi_{j}}|^{N-4}}\right)\,dx+o(\varepsilon)\\
				=&\alpha_{N}^{p}(\delta_{i}\delta_{j})^{\frac{N-4}{2}}\left(\int_{B_{\frac{\eta}{\delta_{i}}}(0)}a(\delta y+\xi_{i}) \frac{1}{(1+|y|^{2})^{\frac{N+4}{2}}}  \left(\frac{1}{|\delta_{i}y+\xi_{i}-\xi_{j}|^{N-4}} -\frac{1}{|\delta_{i}y+\xi_{i}-\widetilde{\xi_{j}}|^{N-4}} \right)\,dy\right)(1+O(\varepsilon))\\
				=&\alpha_{N}^{p}\frac{(\delta_{i}\delta_{j})^{\frac{N-4}{2}}}{\varepsilon^{N-4}}\left(\int_{B_{\frac{\eta}{\delta_{i}}}(0)}a(\xi^{0}_{i}) \frac{1}{(1+|y|^{2})^{\frac{N+4}{2}}}  \left(\frac{1}{|\frac{\delta_{i}}{\varepsilon}y+\frac{\xi_{i}-\xi_{j}}{\varepsilon}|^{N-4}} -\frac{1}{|\frac{\delta_{i}}{\varepsilon}y+\frac{\xi_{i}-\widetilde{\xi_{j}}}{\varepsilon}|^{N-4}} \right)\,dy\right)\left(1+O(\varepsilon)\right)\\
				=&\alpha_{N}^{p}\frac{(\delta_{i}\delta_{j})^{\frac{N-4}{2}}}{\varepsilon^{N-4}}\left(\int_{B_{\frac{\eta}{\delta_{i}}}(0)}a(\xi^{0}_{i}) \frac{1}{(1+|y|^{2})^{\frac{N+4}{2}}}  \left(\frac{1}{|\frac{\xi_{i}-\xi_{j}}{\varepsilon}|^{N-4}} -\frac{1}{|\frac{\xi_{i}-\widetilde{\xi_{j}}}{\varepsilon}|^{N-4}} \right)\,dy\right)(1+O(\varepsilon)).
			\end{aligned}
		\end{equation}
		
		\bigskip

		From here, if $\xi_{i}^{0}\neq \xi_{j}^{0}$, then
		\begin{equation*}
			\begin{aligned}
				\frac{1}{\left|\frac{\xi_{i}-\xi_{j}}{\varepsilon}\right|^{N-4}} =&\frac{1}{\left|\frac{\xi_{i}^{0}-\xi_{j}^{0}}{\varepsilon}+t_{i}\nu(\xi_{i}^{0})+t_{j}\nu(\xi_{j}^{0})\right|^{N-4}} \\
				=&\frac{1}{\left|\frac{\xi_{i}^{0}-\xi_{j}^{0}}{\varepsilon}\right|^{N-4}}-\frac{N-4}{2}\frac{1}{\left|\frac{\xi_{i}^{0}-\xi_{j}^{0}}{\varepsilon}\right|^{N-2}}2\left(\frac{\xi_{i}^{0}-\xi_{j}^{0}}{\varepsilon}\cdot \left(t_{i}\nu(\xi_{i}^{0})+t_{j}\nu(\xi_{j}^{0})\right)\right)+h.o.t. \\
				=&\frac{1}{\left|\frac{\xi_{i}^{0}-\xi_{j}^{0}}{\varepsilon}\right|^{N-4}}+O\left(\varepsilon^{N-3}\right)\frac{1}{\left|\xi_{i}^{0}-\xi_{j}^{0}\right|^{N-3}}+h.o.t. \\
			\end{aligned}
		\end{equation*}
		Similarly, it is proven that
		\begin{equation*}
			\begin{aligned}
				\frac{1}{\left|\frac{\xi_{i}-\widetilde{\xi_{j}}}{\varepsilon}\right|^{N-4}}  =\frac{1}{\left|\frac{\xi_{i}^{0}-\xi_{j}^{0}}{\varepsilon}\right|^{N-4}}+O\left(\varepsilon^{N-3}\right)\frac{1}{\left|\xi_{i}^{0}-\xi_{j}^{0}\right|^{N-3}}+h.o.t. \\
			\end{aligned}
		\end{equation*}
		Hence, we have
		\begin{equation*}
			\begin{aligned}
				&	\dfrac{1}{\left|\frac{\xi_{i}-\xi_{j}}{\varepsilon}\right|^{N-4}} -\frac{1}{\left|\frac{\xi_{i}-\widetilde{\xi_{j}}}{\varepsilon}\right|^{N-4}}=O\left(\varepsilon^{N-3}\right)\frac{1}{\left|\xi_{i}^{0}-\xi_{j}^{0}\right|^{N-3}}+h.o.t. 
			\end{aligned}
		\end{equation*}
		
		\medskip
		Therefore, from \eqref{estp1}, we obtain
		\begin{equation*}
			\begin{aligned}
				&\int_{B_{\eta}(\xi_{i})}a(x)\Delta^{2}(PU_{\delta_{i},\xi_{i}})PU_{\delta_{j},\xi_{j}}\,dx\\
				=&\alpha_{N}^{p}\frac{(\delta_{i}\delta_{j})^{\frac{N-4}{2}}}{\varepsilon^{N-4}}\left(\int_{B_{\frac{\eta}{\delta_{i}}}(0)}a(\xi^{0}_{i}) \frac{1}{\left(1+\left|y\right|^{2}\right)^{\frac{N+4}{2}}}  \left(\frac{1}{\left|\frac{\xi_{i}-\xi_{j}}{\varepsilon}\right|^{N-4}} -\frac{1}{\left|\frac{\xi_{i}-\widetilde{\xi_{j}}}{\varepsilon}\right|^{N-4}} \right)\,dy\right)(1+O(\varepsilon))\\
				=&\alpha_{N}^{p}\dfrac{\left(\delta_{i}\delta_{j}\right)^{\frac{N-4}{2}}}{\varepsilon^{N-4}}\varepsilon^{N-3}\,O\left( \int_{\mathbb{R}^{N}}\frac{1}{\left(1+\left|y\right|^{2}\right)^{\frac{N+4}{2}}}  \,dy \right) \\
			=&o(\varepsilon).
			\end{aligned}
		\end{equation*}
		
		\vspace*{5mm }
		Thus, equation \eqref{inj} is proven. Now we continue with the proof of estimate \eqref{isj}.  Let us assume then that  $\xi_{i}^{0}=\xi_{j}^{0}$, then
		\begin{equation*}
			\begin{aligned}
				\frac{1}{\left|\frac{\xi_{i}-\xi_{j}}{\varepsilon}\right|^{N-4}}-	\frac{1}{\left|\frac{\xi_{i}-\widetilde{\xi_{j}}}{\varepsilon}\right|^{N-4}}&=\frac{1}{\left|\frac{\xi_{i}^{0}-\xi_{i}^{0}}{\varepsilon}+t_{i}\nu(\xi_{i}^{0})-t_{j}\nu(\xi_{i}^{0})\right|^{N-4}}-\frac{1}{\left|\frac{\xi_{i}^{0}-\xi_{i}^{0}}{\varepsilon}+t_{i}\nu(\xi_{i}^{0})+t_{j}\nu(\xi_{i}^{0})\right|^{N-4}}\\ =&\frac{1}{\left|\frac{\xi_{i}^{0}-\xi_{i}^{0}}{\varepsilon}+(t_{i}-t_{j})\nu(\xi_{i}^{0})\right|^{N-4}}-\frac{1}{\left|\frac{\xi_{i}^{0}-\xi_{i}^{0}}{\varepsilon}+(t_{i}+t_{j})\nu(\xi_{i}^{0})\right|^{N-4}} \\
				=&\frac{1}{\left|t_{i}-t_{j}\right|^{N-4}}-\frac{1}{\left|t_{i}+t_{j}\right|^{N-4}} +o(1)\\
			\end{aligned}
		\end{equation*}
		By substituting this equation into equation \eqref{estp1}, we arrive at: 
		\begin{equation}
			\begin{aligned}
				&\int_{B_{\eta}(\xi_{i})}a(x)\Delta^{2}(PU_{\delta_{i},\xi_{i}})PU_{\delta_{j},\xi_{j}}\,dx\\
				=&\alpha_{N}^{p}\frac{(\delta_{i}\delta_{j})^{\frac{N-4}{2}}}{\varepsilon^{N-4}}\left(\int_{B_{\frac{\eta}{\delta_{i}}}(0)}a(\xi^{0}_{i}) \frac{1}{(1+|y|^{2})^{\frac{N+4}{2}}}  \left(\frac{1}{\left|\frac{\xi_{i}-\xi_{j}}{\varepsilon}\right|^{N-4}} -\frac{1}{\left|\frac{\xi_{i}-\widetilde{\xi_{j}}}{\varepsilon}\right|^{N-4}} \right)\,dy\right)(1+O(\varepsilon))\\
				=&\alpha_{N}^{p}\frac{(\delta_{i}\delta_{j})^{\frac{N-4}{2}}}{\varepsilon^{N-4}}\left(\int_{B_{\frac{\eta}{\delta_{i}}}(0)}a(\xi^{0}_{i}) \frac{1}{(1+|y|^{2})^{\frac{N+4}{2}}}  \left(\frac{1}{\left|t_{i}-t_{j}\right|^{N-4}}-\frac{1}{\left|t_{i}+t_{j}\right|^{N-4}} \right)\,dy\right)(1+O(\varepsilon))\\
				=&a(\xi^{0}_{i})\,\varepsilon\,(d_{i}d_{j})^{\frac{N-4}{2}}\left(\frac{1}{|t_{i}-t_{j}|^{N-4}}-\frac{1}{|t_{i}+t_{j}|^{N-4}} \right)\gamma_{2} +o(\varepsilon).\\
			\end{aligned}
		\end{equation}
		This concludes the proof of the Proposition.
	\end{proof}

	\bigskip\bigskip
	
	\section{Error Estimates}\label{sect:Error}

	This section focuses on estimating the leading terms of the error 	$E_{(\boldsymbol{\xi},\textbf{d}, \textbf{t})}$ defined in \eqref{Error}. 

\medskip

    We begin by presenting some preliminary estimates. As before, consider the bubble:
	\begin{equation*}
		U_{\delta,\xi}(x):=\alpha_{N}\left(\frac{\delta}{\delta^{2}+|x-\xi|^{2}} \right)^{\frac{N-4}{2}},\ \  x,\xi\in \mathbb{R}^{N}, \delta\in \mathbb{R}^{+} 
	\end{equation*}
	then,
	\begin{equation*}
		\begin{aligned}
			|U_{\delta,\xi}|_{q}^{q}&=\alpha_{N}^{q}\int_{\Omega}\left(\frac{\delta}{\delta^{2}+|x-\xi|^{2}} \right)^{\frac{N-4}{2}q}\,dx\\
			&=\alpha_{N}^{q}\int_{\frac{\Omega-\xi}{\delta}}\frac{\delta^{\frac{N-4}{2}q+N}}{\delta^{(N-4)q}\left( 1+|y|^{2}\right) ^{\frac{N-4}{2}q}}\,dy =O\left(\delta^{-\frac{N-4}{2}q+N}\right).
		\end{aligned}
	\end{equation*}
	Hence we have:
	\begin{equation*}
		\left|U_{\delta,\xi}\right|_{q}=O(\delta^{\frac{N}{q}-\frac{N-4}{2}}), \ \ \ \ \ \ \mbox{ if } q>\frac{N}{N-4},
	\end{equation*}
	and furthermore, since
	\begin{equation*}
		\begin{aligned}
			\left|U_{\delta,\xi}^{\frac{N+4}{N-4}}\right|_{q}^{q}&=\alpha_{N}^{q}\int_{\Omega}\left(\frac{\delta}{\delta^{2}+|x-\xi|^{2}} \right)^{\frac{N-4}{2}\left(\frac{N+4}{N-4}q\right)}\,dx\\
			&=\alpha_{N}^{q}\int_{\frac{\Omega-\xi}{\delta}}\frac{\delta^{\frac{N+4}{2}q+N}}{\delta^{(N+4)q}\left( 1+|y|^{2}\right) ^{\frac{N+4}{2}q}}\,dy =O\left(\delta^{-\frac{N+4}{2}q+N}\right),
		\end{aligned}
	\end{equation*}
	we obtain
	\begin{equation*}
		\left|U_{\delta,\xi}^{\frac{N+4}{N-4}}\right|_{q}=O\left(\delta^{\frac{N}{q}-\frac{N+4}{2}}\right), \ \ \ \ \ \ \mbox{ if } q\geq 1.
	\end{equation*}

\medskip
    
Next we estimate some terms related to the derivatives of the bubble. 
	
	\begin{proposition}\label{tmij}
		Let $\xi_{i},\xi_{j}\in\Omega$ and $\eta$ defined as in Remark \ref{Rmk2}. The following estimates hold:
		\begin{equation}\label{est1}
			\int_{\Omega}\nabla\left( \Delta(PU_{\delta_{i},\xi_{i}})\right)PU_{\delta_{j},\xi_{j}}\,dx=O\left(\frac{\delta_{i}}{\delta_{j}}\right)^{\frac{N-4}{2}} O\left(\delta_{j}^{\frac{N-4}{N-3}+\beta_{1}}\right),
		\end{equation}
		and
		\begin{equation}\label{est2}
			\int_{\Omega}\Delta(PU_{\delta_{i},\xi_{i}})PU_{\delta_{j},\xi_{j}}\,dx=O\left(\frac{\delta_{i}}{\delta_{j}}\right)^{\frac{N-4}{2}} O\left(\delta_{j}^{\frac{N-4}{N-3}+\beta_{2}}\right).
		\end{equation}
		Moreover, we also have
		\begin{equation}\label{est3}
			\left|  \nabla_{x}\left( \Delta\left(PU_{\delta_{i},\xi_{i}}\right)\right)\right| _{\frac{2N}{N+4}} =O\left(\delta^{\frac{N-4}{2(N-3)}+\beta}\right),
		\end{equation}
		and
		\begin{equation}\label{est4}
			\left| \Delta\left(PU_{\delta_{i},\xi_{i}}\right)\right| _{\frac{2N}{N+4}} =O\left(\delta^{\frac{N-4}{2(N-3)}+\beta}\right).
		\end{equation}
		
	\end{proposition}
	\bigskip
	\begin{proof}
		
		As the function $(\Delta PU_{\delta_{i},\xi_{i}})$ satisfies the problem:
		\begin{equation}
			\left\{
			\begin{array}
				[c]{ll}%
				\Delta(\Delta PU_{\delta_{i},\xi_{i}}) =U_{\delta_{i},\xi_{i}}^{\frac{N+4}{N-4}} & \text{in }\Omega,\\
				\Delta PU_{\delta_{i},\xi_{i}}=0&  \text{on }\partial\Omega
			\end{array}
			\right.  %
		\end{equation}
		we have that:
		\begin{equation*}
			-\Delta(PU_{\delta_{i},\xi_{i}})(x)=	\int_{\Omega}G_{\Delta}(x,y)\,U_{\delta_{i},\xi_{i}}(y)^{\frac{N+4}{N-4}}\,dy
		\end{equation*}
		and
		\begin{equation*}
			-	\nabla_{x}\left( \Delta(PU_{\delta_{i},\xi_{i}})\right)(x)=	\int_{\Omega}\nabla_{x}(G_{\Delta}(x,y))\,U_{\delta_{i},\xi_{i}}(y)^{\frac{N+4}{N-4}}\,dy,
		\end{equation*}
		where $G_{\Delta}(x,y)$ is the Green's function for the Laplacian with Dirichlet boundary condition.
		
		We begin by proving estimate  \eqref{est1}.	 Choose an element $r'\in \left(N, \frac{N(N-3)}{N-4}\right)$. If $r$  is such that  $\frac{1}{r}+\frac{1}{r'}+1=2$ then $r\in \left(\frac{N(N-3)}{(N-2)^{2}}, \frac{N}{N-1}\right)$.  
		For these values of $r$ and $r'$, we have the following implications:
		\begin{equation*}
			\begin{aligned}
				r<\frac{N}{N-1} &\implies |Q_{1}(x)|_{\Omega,r}:=\int_{\Omega}\frac{1}{|x|^{r(N-1)}}\,dx<\infty\\
				r>\frac{N(N-3)}{(N-2)^{2}}\geq 1 &\implies
				\left|U_{\delta,\xi}^{\frac{N+4}{N-4}}\right|_{1}=O\left(\delta^{N-\frac{N+4}{2}}\right)\\
				r'<\frac{N(N-3)}{N-4}&\implies  \frac{N}{r'}>\frac{N-4}{N-3} \\
				r'>N\geq\frac{N}{N-4} &\implies  \left|U_{\delta,\xi}\right|_{r'}=O\left(\delta^{\frac{N}{r'}-\frac{N-4}{2}}\right). \\
			\end{aligned}
		\end{equation*}

		To estimate \eqref{est1} we rely on Young's inequality (see Theorem 4.2 in \cite{lieb2001analysis}). Using \eqref{G2}, we obtain:
		\begin{equation*}
			\begin{aligned}
				\left|\int_{\Omega}\nabla\left( \Delta(PU_{\delta_{i},\xi_{i}})\right)PU_{\delta_{i},\xi_{i}}\,dx\right|&\leq 	\int_{\Omega}\int_{\Omega}\left|\nabla_{x}(G_{\Delta}(x,y))U_{\delta_{i},\xi_{i}}(y)^{\frac{N+4}{N-4}}PU_{\delta_{j},\xi_{j}}(x)\right|\,dy\,dx\\
				&\leq C	\int_{\Omega}\int_{\Omega}\left|\frac{1}{|x-y|^{N-1}}U_{\delta_{i},\xi_{i}}(y)^{\frac{N+4}{N-4}}PU_{\delta_{j},\xi_{j}}(x)\right|\,dy\,dx\\
				&\leq C\left|Q_{1}(x)\right|_{\Omega,r}\left|U_{\delta_{i},\xi_{i}}(y)^{\frac{N+4}{N-4}}\right|_{1}\left|U_{\delta_{j},\xi_{j}}(x)\right|_{r'}\\
				&=O\left(\delta_{i}^{N-\frac{N+4}{2}}\right)O\left(\delta_{j}^{\frac{N}{r'}-\frac{N-4}{2}}\right)\\
				&=O\left(\frac{\delta_{i}}{\delta_{j}}\right)^{\frac{N-4}{2}} O\left(\delta_{j}^{\frac{N}{r'}}\right)\\
				&=O\left(\frac{\delta_{i}}{\delta_{j}}\right)^{\frac{N-4}{2}} O\left(\delta_{j}^{\frac{N-4}{N-3}+\beta}\right),
			\end{aligned}
		\end{equation*}
		for some $\beta>0$.
		\bigskip
		
		Now, we continue with the proof of equation \eqref{est2}. We first consider the case where $N\geq 6$. Choose an element $r'$ such that $r'\in \left(\frac{N}{2}, \frac{N(N-3)}{N-4}\right)$. If $r$  is such that  $\frac{1}{r}+\frac{1}{r'}+1=2$ then   $r\in\left(\frac{N(N-3)}{(N-2)^{2}},\frac{N}{N-2}\right)$. For these $r$ and $r'$, the following implications hold:
		\begin{equation*}
			\begin{aligned}
				r<\frac{N}{N-2} &\implies \left|Q_{2}(x)\right|_{\Omega,r}:=\int_{\Omega}\frac{1}{|x|^{r(N-2)}}\,dx<\infty\\
				r>\frac{N(N-3)}{(N-2)^{2}}\geq 1 &\implies
				\left|U_{\delta,\xi}^{\frac{N+4}{N-4}}\right|_{1}=O\left(\delta^{N-\frac{N+4}{2}}\right)\\
				r'<\frac{N(N-3)}{N-4}&\implies  \frac{N}{r'}>\frac{N-4}{N-3} \\
				r'>\frac{N}{2}\geq\frac{N}{N-4} &\implies  \left|U_{\delta,\xi}\right|_{r'}=O\left(\delta^{\frac{N}{r'}-\frac{N-4}{2}}\right).
			\end{aligned}
		\end{equation*}
		For example, if  $N=5$, $r'\in(5,10)$ and  $r$  is such that  $\frac{1}{r}+\frac{1}{r'}+1=2$,  then $r\in\left(\frac{10}{9}, \frac{5}{4}\right)$. Similar implications can be derived easily.

		To estimate \eqref{est2} we use again Young's inequality,  and apply  \eqref{G1} to obtain:
		\begin{equation*}
			\begin{aligned}
				\left|\int_{\Omega}\Delta(PU_{\delta_{i},\xi_{i}})PU_{\delta_{j},\xi_{j}} \,dx\right|&\leq\int_{\Omega}\int_{\Omega}\left|G_{\Delta}(x,y)U_{\delta_{i},\xi_{i}}(y)^{\frac{N+4}{N-4}}PU_{\delta_{j},\xi_{j}}(x)\right|\,dy\,dx\\
				&\leq C\int_{\Omega}\int_{\Omega}\frac{1}{|x-y|^{n-2}}U_{\delta_{i},\xi_{i}}(y)^{\frac{N+4}{N-4}}PU_{\delta_{j},\xi_{j}}(x)\,dy \,dx\\
				&\leq C\left|Q_{2}(x)\right|_{\Omega,r}|U_{\delta_{i},\xi_{i}}(y)^{\frac{N+4}{N-4}}|_{1}\left|U_{\delta_{j},\xi_{j}}(x)\right|_{r'}\\
				&=O(\delta_{i}^{\frac{N}{1}-\frac{N+4}{2}})O\left(\delta_{j}^{\frac{N}{r'}-\frac{N-4}{2}}\right)\\
				&=O\left(\frac{\delta_{i}}{\delta_{j}}\right)^{\frac{N-4}{2}} O\left(\delta_{j}^{\frac{N}{r'}}\right)\\
				&=O\left(\frac{\delta_{i}}{\delta_{j}}\right)^{\frac{N-4}{2}} O\left(\delta_{j}^{\frac{N-4}{N-3}+\beta}\right),
			\end{aligned}
		\end{equation*}
		for some $\beta>0$.
		
		\bigskip
		
		Now we give the proof of estimation \eqref{est3}.	 For $N\geq 6$, choose an element $r'\in \left(\frac{2N}{N+6}, \frac{2N(N-3)}{N^{2}+2N-16}\right)$. If $r$  is such that  $\frac{1}{r}+\frac{1}{r'}=1+\frac{N+4}{2N}$ then $r\in \left[1, \frac{N}{N-1}\right)$.  
		Hence, for these $r$ and $r'$, we have the following implications:
		\begin{equation*}
			\begin{aligned}
				r<\frac{N}{N-1} &\implies |Q_{1}(x)|_{\Omega,r}:=\int_{\Omega}\frac{1}{|x|^{r(n-1)}}\,dx<\infty\\
				r'<\frac{2N(N-3)}{N^{2}+2N-16}&\implies  \frac{N}{r'}-\frac{N+4}{2}>\frac{N-4}{2(N-3)} \\
				r'>\frac{2N}{N+6}\geq 1 &\implies  \left|U_{\delta,\xi}^{\frac{N+4}{N-4}}\right|_{r'}=O\left(\delta^{\frac{N}{r'}-\frac{N+4}{2}}\right). \\
			\end{aligned}
		\end{equation*}
		
		If $5\leq N<6$, choose $r=\frac{2N}{N+4}$ and $r'=1$. Then 
		\begin{equation*}
			\begin{aligned}
				r<\frac{N}{N-1} &\implies |Q_{1}(x)|_{\Omega,r}:=\int_{\Omega}\frac{1}{|x|^{r(N-1)}}\,dx<\infty\\
				r'=1&\implies  \frac{N}{r'}-\frac{N+4}{2}>\frac{N-4}{2(N-3)} \\
				r'= 1 &\implies  \left|U_{\delta,\xi}^{\frac{N+4}{N-4}}\right|_{r'}=O\left(\delta^{\frac{N}{r'}-\frac{N+4}{2}}\right). \\
			\end{aligned}
		\end{equation*}
		To prove  \eqref{est3}, we use  Remark $(2)$ from Theorem 4.2 in \cite{lieb2001analysis} to obtain:
		\begin{equation*}
			\begin{aligned}
				\left|  \nabla_{x}\left( \Delta\left(PU_{\delta_{i},\xi_{i}}\right)\right)\right| _{\frac{2N}{N+4}} &=\left|\int_{\Omega}\nabla_{x}(G_{\Delta}(x,y))U_{\delta_{i},\xi_{i}}(y)^{\frac{N+4}{N-4}}\,dy\right|_{\frac{2N}{N+4}}\\
				&\leq C \left|Q_{1}(x)\right|_{\Omega,r}\left|U_{\delta,\xi}^{\frac{N+4}{N-4}}\right|_{r'}=O\left(\delta^{\frac{N}{r'}-\frac{N+4}{2}}\right) =O\left(\delta^{\frac{N-4}{2(N-3)}+\beta}\right),
			\end{aligned}
		\end{equation*}
		for any $N>4$.
		
		\bigskip
		
		The proof of estimation \eqref{est4} is obtained in a similar manner:
		If $N\geq 8$, choose an element $r'\in (\frac{2N}{N+8}, \frac{2N(N-3)}{N^{2}+2N-16})$. Let $r$  be such that  $\frac{1}{r}+\frac{1}{r'}=1+\frac{N+4}{2N}$,  then $r\in [1, \frac{N}{N-2})$.  
		For these choices of  $r$ and $r'$, we have the following implications:
		\begin{equation*}
			\begin{aligned}
				r<\frac{N}{N-2} &\implies \left|Q_{2}(x)\right|_{\Omega,r}:=\int_{\Omega}\frac{1}{|x|^{r(N-2)}}\,dx<\infty\\
				r'<\frac{2N(N-3)}{N^{2}+2N-16}&\implies  \frac{N}{r'}-\frac{N+4}{2}>\frac{N-4}{2(N-3)} \\
				r'>\frac{2N}{N+8}\geq 1 &\implies  \left|U_{\delta,\xi}^{\frac{N+4}{N-4}}\right|_{r'}=O\left(\delta^{\frac{N}{r'}-\frac{N+4}{2}}\right). \\
			\end{aligned}
		\end{equation*}
		On the other hand, if $5\leq N<8$, choose $r=\frac{2N}{N+4}$ and $r'=1$. Then 
		\begin{equation*}
			\begin{aligned}
				r<\frac{N}{N-2} &\implies |Q_{2}(x)|_{\Omega,r}:=\int_{\Omega}\frac{1}{|x|^{r(N-2)}}\,dx<\infty\\
				r'=1&\implies  \frac{N}{r'}-\frac{N+4}{2}>\frac{N-4}{2(N-3)} \\
				r'= 1 &\implies  \left|U_{\delta,\xi}^{\frac{N+4}{N-4}}\right|_{r'}=O\left(\delta^{\frac{N}{r'}-\frac{N+4}{2}}\right). \\
			\end{aligned}
		\end{equation*}
		
		We use the Remark $(2)$ from Theorem 4.2 in \cite{lieb2001analysis},  to obtain:
		\begin{equation*}
			\begin{aligned}
				\left| \Delta(PU_{\delta_{i},\xi_{i}})\right| _{\frac{2N}{N+4}}&= \left|\int_{\Omega}G_{\Delta}(x,y)U_{\delta_{i},\xi_{i}}(y)^{\frac{N+4}{N-4}}\,dy \right| _{\frac{2N}{N+4}}\\
				&\leq C |Q_{2}(x)|_{\Omega,r}\left|U_{\delta,\xi}^{\frac{N+4}{N-4}}\right|_{r'}=O\left(\delta^{\frac{N}{r'}-\frac{N+4}{2}}\right) =O\left(\delta^{\frac{N-4}{2(N-3)}+\beta}\right),
			\end{aligned}
		\end{equation*}
		for any $N>4$.  This concludes the proof of the proposition.
		
	\end{proof}
	
	\bigskip
	
	With the above estimates in place, we are now in a position to estimate the error term $E_{(\boldsymbol{\xi},\textbf{d}, \textbf{t})}$ . Our attention will be directed toward the estimation in case \eqref{typ2}, as the estimation in case \eqref{typ1} is more straightforward and can be directly obtained from the former.
	Let $(\zeta^{0}, \textbf{d}, \textbf{t})\in \Sigma_{2}$ and, as before, let 
	\begin{equation*}
		\eta=\min\left\{d(\xi_{1}, \partial\Omega), d(\xi_{2}, \partial\Omega), \frac{|\xi_{1}-\xi_{2}|}{2} \right\}=\min\left\{t_{1}\varepsilon,t_{2}\varepsilon,\frac{\varepsilon|t_{1}-t_{2}|}{2} \right\}.
	\end{equation*}
	\begin{lemma}\label{error}
		It holds true for some $\sigma > 0$ that
		\[
		\|E_{(\zeta^0,\textbf{d}, \textbf{t})}\| = O\left( \varepsilon^{\frac{1}{2} + \sigma} \right),
		\]
		on compact subsets of $\Sigma_{2}$.
	\end{lemma}

\medskip
    
	\begin{proof}
		
		We begin by evaluating the error term $E_{(\xi^{0},\textbf{d}, \textbf{t})}$:
		\begin{equation*}
			\begin{aligned}
				\|E_{(\xi^{0},\textbf{d}, \textbf{t})}\| &= \| \Delta(a(x)\Delta V_{(\xi^{0},\textbf{d}, \textbf{t})}) - a(x)|V_{(\xi^{0},\textbf{d}, \textbf{t})}|^{p-2-\varepsilon}V_{(\xi^{0},\textbf{d}, \textbf{t})}\|\\
				&\leq C\left| \Delta(a(x)\Delta V_{(\xi^{0},\textbf{d}, \textbf{t})}) - a(x)|V_{(\xi^{0},\textbf{d}, \textbf{t})}|^{p-2-\varepsilon}V_{(\xi^{0},\textbf{d}, \textbf{t})}\right|_{\frac{2N}{N+4}} \\
				&\leq C \left| \Delta a\Delta V_{(\xi^{0},\textbf{d}, \textbf{t})}+2\nabla a \nabla \left(\Delta V_{(\xi^{0},\textbf{d}, \textbf{t})} \right)  + a(x) \Delta^{2} V_{(\xi^{0},\textbf{d}, \textbf{t})} - a(x) f_\varepsilon(V_{(\xi^{0},\textbf{d}, \textbf{t})}) \right|_{\frac{2N}{N+4}} \\
				&\leq O \left( \sum_i |\nabla a \Delta PU_i|_{\frac{2N}{N+4}} \right) +O\left( \sum_i |\nabla a \nabla(\Delta PU_i)|_{\frac{2N}{N+4}} \right) \\
				&+ O\left( \left|  a(x) \Delta^{2} V_{(\xi^{0},\textbf{d}, \textbf{t})} - a(x) f_\varepsilon(V_{(\xi^{0},\textbf{d}, \textbf{t})}) \right|_{\frac{2N}{N+4}} \right).
			\end{aligned}
		\end{equation*}

		\bigskip
		
		By Lemma \ref{tmij} (estimates \eqref{est3} and \eqref{est4}), and using that $\delta_i = O\left(\varepsilon^{\frac{n-1}{n-2}}\right)$ on compact subsets of $\Sigma$, we obtain:
		\begin{equation}
			\left|  \nabla_{x}\left( \Delta(PU_{\delta_{i},\xi_{i}})\right)\right| _{\frac{2N}{N+4}} =O\left(\delta_{i}^{\frac{N-4}{2(N-3)}+\beta}\right)=O\left(\left(\varepsilon^{\frac{N-3}{N-4}}\right)^{\frac{N-4}{2(N-3)}+\beta}\right),
		\end{equation}
		and
		\begin{equation}
			\left| \Delta(PU_{\delta_{i},\xi_{i}})\right| _{\frac{2N}{N+4}} =O\left(\delta^{\frac{N-4}{2(N-3)}+\beta}\right)=O\left(\left(\varepsilon^{\frac{N-3}{N-4}}\right)^{\frac{N-4}{2(N-3)}+\beta}\right).
		\end{equation}
		Therefore:
		\begin{equation*}
			\left|\nabla a \Delta PU_i\right|_{\frac{2N}{N+4}}=O\left( \varepsilon^{\frac{1}{2} + \sigma} \right) \mbox{ and } |\nabla a \nabla(\Delta PU_i)|_{\frac{2N}{N+4}}=O\left( \varepsilon^{\frac{1}{2} + \sigma} \right),
		\end{equation*}
		for some $\sigma>0$. In the remainder of the proof, we will estimate the remaining term:
		\begin{equation*}
			\begin{aligned}
				\left|  a(x) \Delta^{2} V_{(\xi^{0},\textbf{d}, \textbf{t})} - a(x) f_\varepsilon(V_{(\xi^{0},\textbf{d}, \textbf{t})}) \right|_{\frac{2N}{N+4}} &\leq 	\left|  a(x) \Delta^{2} V_{(\xi^{0},\textbf{d}, \textbf{t})} - a(x) f_{0}(V_{(\xi^{0},\textbf{d}, \textbf{t})}) \right|_{\frac{2N}{N+4}}\\
				&+\left| a(x) f_{0}(V_{(\xi^{0},\textbf{d}, \textbf{t})})  - a(x) f_\varepsilon(V_{(\xi^{0},\textbf{d}, \textbf{t})}) \right|_{\frac{2N}{N+4}}\\
				&=:B_{1}+B_{2}.\\
			\end{aligned}
		\end{equation*}	
		First, we will focus on the term $B_{1}$. Indeed, 
		\begin{equation*}
			\begin{aligned}
				&\left|   a(x) \Delta^{2} V_{(\xi^{0},\textbf{d}, \textbf{t})} - a(x) f_{0}(V_{(\xi^{0},\textbf{d}, \textbf{t})})  \right|_{\frac{2N}{N+4}}=\left|  a(x) \left(U_{1}^{p-1}-U_{2}^{p-1} \right)  - a(x) (PU_{1}-PU_{2})^{p-1}\right|_{\frac{2N}{N+4}}\\
				&=O\left( \left|     \left(U_{1}^{p-1}-U_{2}^{p-1} \right)  -   (PU_{1}-PU_{2})^{p-1}  \right|_{\Omega\smallsetminus  (B_{\eta}(\xi_{1})\cup B_{\eta}(\xi_{2})), \frac{2N}{N+4}}\right) \\
				&+\sum_{i=1}^{2}O\left(\left|     \left(U_{1}^{p-1}-U_{2}^{p-1} \right)  -   (PU_{1}-PU_{2})^{p-1}  \right|_{B_{\eta}(\xi_{i}), \frac{2N}{N+4}} \right).
			\end{aligned}
		\end{equation*}	
		The estimate in the set $\Omega\smallsetminus  (B_{\eta}(\xi_{1})\cup B_{\eta}(\xi_{2}))$ is obtained by:
		\begin{equation*}
			\begin{aligned}
				&\left|     \left(U_{1}^{p-1}-U_{2}^{p-1} \right)  -   (PU_{1}-PU_{2})^{p-1} \right|_{\Omega\smallsetminus  (B_{\eta}(\xi_{1})\cup B_{\eta}(\xi_{2})), \frac{2N}{N+4}} \\
				&\leq \left|  U_{1}^{p-1} \right|_{\Omega\smallsetminus  (B_{\eta}(\xi_{1})), \frac{2N}{N+4}}+ \left|  U_{2}^{p-1} \right|_{\Omega\smallsetminus  (B_{\eta}(\xi_{2})), \frac{2N}{N+4}}\\
				&=O\left( \left( \frac{\delta_{1}}{\eta}\right)^{N} \right)+O\left( \left( \frac{\delta_{2}}{\eta}\right)^{N} \right),
			\end{aligned}
		\end{equation*}	
		because 
		\begin{equation}\label{ecuu2}
			\begin{aligned}
				\left|  U_{i}^{p-1} \right|_{\Omega\smallsetminus  (B_{\eta}(\xi_{i})), \frac{2N}{N+4}} &=\int_{\Omega\smallsetminus  (B_{\eta}(\xi_{i})) }	\frac{\delta_{i}^{N}}{(\delta_{i}^{2}+|x-\xi_{i}|^{2})^{N} }\,dx\\
				&\leq \int_{\{ |y|>\frac{\eta}{\delta_{i}}\}} \frac{1}{(1+|y|^{2})^{N} }\,dy\\
				&=O\left( \left( \frac{\delta_{i}}{\eta}\right)^{N} \right).
			\end{aligned}
		\end{equation}
		
		\bigskip
		
		The other terms that form $B_{1}$ are addressed as follows:
		\begin{equation*}
			\begin{aligned}
				&\left|    \left(U_{1}^{p-1}-U_{2}^{p-1} \right)  -   (PU_{1}-PU_{2})^{p-1}  \right|_{B_{\eta}(\xi_{1}), \frac{2N}{N+4}}\\
				\leq&  \left| (PU_{1}-PU_{2})^{p-1}-  a(x)  U_{1}^{p-1} \right|_{B_{\eta}(\xi_{1}), \frac{2N}{N+4}}+\left|  U_{2}^{p-1} \right|_{B_{\eta}(\xi_{1}), \frac{2N}{N+4}}\\
				=&  \left|  (U_{1}+PU_{1}-U_{1}-PU_{2})^{p-1}-  a(x)  U_{1}^{p-1} \right|_{B_{\eta}(\xi_{1}), \frac{2N}{N+4}}+\left|  U_{2}^{p-1} \right|_{B_{\eta}(\xi_{1}), \frac{2N}{N+4}}\\
				=&  \left|   (p-1)(U_{1}+\theta(PU_{1}-U_{1}-PU_{2}))^{p-2}(PU_{1}-U_{1}-PU_{2}) \right|_{B_{\eta}(\xi_{1}), \frac{2N}{N+4}}+\left|  U_{2}^{p-1} \right|_{B_{\eta}(\xi_{1}), \frac{2N}{N+4}}\\
				= & O\left(  \left|  U_{1}^{p-2}(PU_{1}-U_{1})  \right|_{B_{\eta}(\xi_{1}), \frac{2N}{N+4}}\right) +O\left( \left|U_{1}^{p-2}PU_{2} \right|_{B_{\eta}(\xi_{1}), \frac{2N}{N+4}}\right) \\
				+ & O\left(  \left| (PU_{1}-U_{1})^{p-1} \right|_{B_{\eta}(\xi_{1}), \frac{2N}{N+4}}\right) +O\left(\left|(PU_{1}-U_{1})^{p-2}PU_{2} \right|_{B_{\eta}(\xi_{1}), \frac{2N}{N+4}} \right) \\
				+ &  O\left( \left| PU_{2}^{p-2}(PU_{1}-U_{1}) \right|_{B_{\eta}(\xi_{1}), \frac{2N}{N+4}}\right) +O\left( \left| U_{2}^{p-1} \right|_{B_{\eta}(\xi_{1}), \frac{2N}{N+4}}\right) \\
				=:&A_{1}+A_{2}+A_{3}+A_{4}+A_{5}+A_{6}
			\end{aligned}
		\end{equation*}	
		We will show that every term $A_{1},\dots,A_{6} $  is of the order \( O\left( \varepsilon^{\frac{1}{2} + \sigma} \right) \). The calculations are based on the following basic estimates:
		\begin{equation}\label{CRes}
			\begin{aligned}
				\left|  U_{1}^{(p-2)} \right|_{\frac{2N}{N+4}r}^{\frac{2N}{N+4}r} &=\int_{B_{\eta}(\xi_{1}) }	 U_{1}^{(p-2)\left(\frac{2N}{N+4}r\right)}\,dx\\
				&=\int_{B_{\eta}(\xi_{1}) }	\frac{\delta_{1}^{\frac{8Nr}{N+4}}}{\left(\delta_{1}^{2}+|x-\xi_{1}|^{2}\right)^{\frac{8Nr}{N+4}} }\,dx\\
				&=\delta_{1}^{N-\frac{8Nr}{N+4}}\int_{B_{\frac{\eta}{\delta_{1}}}(0) }	\frac{1}{\left(1+|y|^{2}\right)^{\frac{8Nr}{N+4}} }\,dy\\
				&= \left\{ \begin{array}{lcc} \delta_{1}^{N-\frac{8Nr}{N+4}}\left(\frac{\eta}{\delta_{1}} \right)^{N-\frac{16Nr}{N+4}}  & if &N>12 \mbox{ and } 1<r < \frac{N+4}{16} \\ \\ \delta_{1}^{N-\frac{8Nr}{N+4}}  & if & N\leq 12 \mbox{ and } r > \frac{N+4}{16} \end{array} \right.
			\end{aligned}
		\end{equation}
		i.e.
		\begin{equation}\label{U1}
			\begin{aligned}
				\left|  U_{1}^{(p-2)} \right|_{\frac{2N}{N+4}r} =& \left\{ \begin{array}{lcc} \delta_{1}^{4}\eta ^{\frac{N+4}{2r}-8}  & if &N>12 \mbox{ and } 1<r < \frac{N+4}{16} \\ \\ \delta_{1}^{\frac{N+4}{2r}-4}  & if & N\leq 12 \mbox{ and } r > \frac{N+4}{16} \end{array} \right.
			\end{aligned}
		\end{equation}
		
		\bigskip
		
		Additionally, as $\left|\frac{2\tau_{1}\nu(\xi^{0})}{\eta}\right|>1$ there exists $k>0$ such that:
		\begin{equation}\label{PU1}
			\begin{aligned}
				\left|\frac{1}{|x-\widetilde{\xi_{i}}|^{N-4}}\right|_{{\frac{2N r}{(N+4)(r-1)}}}	=&\left(\int_{B_{\eta}(\xi_{i})}\frac{1}{\left( |x-\widetilde{\xi_{i}}|^{N-4}\right)^{\frac{2N r}{(N+4)(r-1)}}} \,dx\right)^{\frac{(N+4)(r-1)}{2Nr}} \\
				=& \left( 	\int_{B_{1}\left(\frac{\xi_{i}-\widetilde{\xi_{i}}}{\eta}\right)}\frac{\eta^{N}}{\left( |\eta y|^{N-4}\right)^{\frac{2N r}{(N+4)(r-1)}}}\,dy\right)^{\frac{(N+4)(r-1)}{2Nr}}\\
				=&\left( \eta^{N-\frac{2(N-4)N r}{(N+4)(r-1)}}\int_{B_{1}\left(\frac{2\tau_{1}\nu(\xi^{0})}{\eta}\right)}\frac{1}{| y|^{\frac{2(N-4)N r}{(N+4)(r-1)}}} \,dy\right)^{\frac{(N+4)(r-1)}{2Nr}} \\
				=&\eta^{\frac{(N+4)(r-1)}{2r}-(N-4)}\left( \int_{|y|>k}\frac{1}{| y|^{\frac{2(N-4)N r}{(N+4)(r-1)}}} \,dy\right)^{\frac{(N+4)(r-1)}{2Nr}} \\
				=&O\left(\eta^{-\frac{N-12}{2}-\frac{N+4}{2r}}\right),
			\end{aligned}
		\end{equation}
		choosing \( 1 < r < \frac{N+4}{12 - N} \) if \( 12 > N \) and \( r > 1 \) if \( 12 \leq N \).  A similar argument yields to: 
		\begin{equation}\label{U2}
			\begin{aligned}
				\left|U_{2} \right|_{B_{\eta}(\xi_{1}),{\frac{2N r}{(N+4)(r-1)}}}\leq& \delta_{2}^{\frac{N-4}{2}}\left(\int_{B_{\eta}(\xi_{1})}\frac{1}{\left( |x-\xi_{2}|^{N-4}\right)^{\frac{2N r}{(N+4)(r-1)}}}\,dx \right)^{\frac{(N+4)(r-1)}{2Nr}} \\
				\leq&  \delta_{2}^{\frac{N-4}{2}}\left( 	\int_{B_{1}\left(\frac{\xi_{1}-\xi_{2}}{\eta}\right)}\frac{\eta^{N}}{\left( |\eta y|^{N-4}\right)^{\frac{2N r}{(N+4)(r-1)}}}\,dy\right)^{\frac{(N+4)(r-1)}{2Nr}}\\
				=&\delta_{2}^{\frac{N-4}{2}}O\left(\eta^{-\frac{N-12}{2}-\frac{N+4}{2r}}\right),
			\end{aligned}
		\end{equation}
		when choosing, as before, \( 1 < r < \frac{N+4}{12 - N} \) if \( 12 > N \) and \( r > 1 \) if \( 12 \leq N \).
		
		\bigskip

		The estimate of \( A_{1} \) is obtained using \eqref{equa2}, along with \eqref{U1} and \eqref{PU1},  choosing $r>1$ such that \( r \sim 1 \) when \( N \leq 12 \) or \( r \sim \frac{N + 4}{16} \) when \( N > 12 \). Indeed,

		\begin{equation*}
			\begin{aligned}
				A_{1}	=&  \left|  U_{1}^{p-2}\left(PU_{1}-U_{1}\right)  \right|_{B_{\eta}(\xi_{1}), \frac{2N}{N+4}}=		\left(\int_{B_{\eta}(\xi_{1}) }	\left| U_{1}^{p-2}(PU_{1}-U_{1}) \right|^{\frac{2N}{N+4}}\,dx\right) ^{\frac{N+4}{2N}}\\
				\leq& \left|  U_{1}^{(p-2)} \right|_{\frac{2N}{N+4}r}\left|PU_{1}-U_{1}\right|_{{\frac{2N r}{(N+4)(r-1)}}}\\
				\leq& \delta_{1}^{\frac{N-4}{2}}\left|  U_{1}^{(p-2)} \right|_{\frac{2N}{N+4}r}\left|\frac{1}{|x-\widetilde{\xi_{i}}|^{N-4}}\right|_{{\frac{2N r}{(N+4)(r-1)}}}\\
				=& \left\{ \begin{array}{lcc}  O\left( \left( \frac{\delta_{1}}{\eta}\right) ^{\frac{N+4}{2}-\sigma}\right)    & if &N>12 \\ \\ O\left( \left( \frac{\delta_{1}}{\eta}\right) ^{N-4-\sigma} \right)  & if & N\leq 12\end{array} \right.
			\end{aligned}
		\end{equation*}	
		Here, $\sigma$ is a positive small constant, with $\sigma \sim (r-1)$  for \( N \leq 12 \) and \(\sigma  \sim  \left(r-\frac{N + 4}{16}\right) \) for \( N > 12 \). 
		
		The term \( A_{2} \) is estimated similarly, this time using the equation \eqref{equa1} along with equations \eqref{U1} and \eqref{U2} and  choosing $r>1$ such that \( r \sim 1 \) when \( N \leq 12 \) or \( r \sim \frac{N + 4}{16} \) when \( N > 12 \). Indeed, 
		\begin{equation*}
			\begin{aligned}
				A_{2}\leq & 	\left|U_{1}^{p-2}PU_{2} \right|_{B_{\eta}(\xi_{1}), \frac{2N}{N+4}}\leq c\left|U_{1}^{p-2}U_{2} \right|_{B_{\eta}(\xi_{1}), \frac{2N}{N+4}} \\
				\leq& c\left|  U_{1}^{(p-2)} \right|_{\frac{2N}{N+4}r}\left|U_{2}\right|_{{\frac{2N r}{(N+4)(r-1)}}}\\
				=& \left\{ \begin{array}{lcc}  \left( \frac{\delta_{1}}{\eta}\right) ^{4-\sigma}\left( \frac{\delta_{2}}{\eta}\right) ^{\frac{N-4}{2}-\sigma}   & if &N>12 \\ \\ \left( \frac{\delta_{1}}{\eta}\right) ^{\frac{N-4-\sigma}{2}} \left( \frac{\delta_{2}}{\eta}\right) ^{\frac{N-4-\sigma}{2}}  & if & N\leq 12\end{array} \right.
			\end{aligned}
		\end{equation*}
		Estimating \( A_{3} \) follows a similar process, using equation \eqref{equa2} of Proposition \ref{equa1} to derive:
		\begin{equation*}
			\begin{aligned}
				A_{3}=	\left| (PU_{1}-U_{1})^{p-1} \right|_{B_{\eta}(\xi_{1}), \frac{2N}{N+4}}&\leq \left|\left(\frac{\delta_{1}^{\frac{N-4}{2}}}{|x-\widetilde{\xi}_{1}|^{N-4}} \right) ^{p-1} \right|_{B_{\eta}(\xi_{1}), \frac{2N}{N+4}}\\
				&\leq \delta_{1}^{\frac{N+4}{2}}  \left(\int_{B_{\eta}(\xi_{1})} \frac{1}{|x-\widetilde{\xi}_{1}|^{2N}}  \,dx\right)^{\frac{N+4}{2N}}\\
				&\leq \delta_{1}^{\frac{N+4}{2}}  \left(	\int_{B_{1}(\frac{\xi_{1}-\widetilde{\xi_{1}}}{\eta})}\frac{\eta^{N}}{ |\eta y|^{2N} } \,dy \right)^{\frac{N+4}{2N}}\\
				&\leq O\left( \delta_{1}^{\frac{N+4}{2}}\eta^{-\frac{N+4}{2}}\right). \\
			\end{aligned}
		\end{equation*}	
		We handle \( A_{4} \) as follows:
		\begin{equation*}
			\begin{aligned}
				A_{4}	\leq &  \left|  (PU_{1}-U_{1})^{p-2}U_{2}  \right|_{B_{\eta}(\xi_{1}), \frac{2N}{N+4}}\\
				\leq& \left|  (PU_{1}-U_{1})^{(p-2)} \right|_{\frac{2N}{N+4}r}\left|U_{2}\right|_{B_{\eta}(\xi_{1}),{\frac{2N r}{(N+4)(r-1)}}}\\
				=& \left\{ \begin{array}{lcc}  O\left( \left( \frac{\delta_{1}}{\eta}\right) ^{\frac{N+4}{2}-\sigma}\right)    & if &N>12 \\ \\ O\left( \left( \frac{\delta_{1}}{\eta}\right) ^{N-4-\sigma} \right)  & if & 5\leq N\leq 12\end{array} \right.
			\end{aligned}
		\end{equation*}	
		Finally, term \( A_{5} \) is similarly estimated as \( A_{1} \), whereas \( A_{6} \) is handled like \eqref{ecuu2}. This concludes the estimation of \( B_{1} \).

		\bigskip
		
		Arguing in the same way as in \cite[Proposition 2]{rey1991blow}, we can estimate the last term \( B_2 \) as
		\[
		\left| a(x) f_{0}(V_{(\xi^{0},\textbf{d}, \textbf{t})})  - a(x) f_\varepsilon(V_{(\xi^{0},\textbf{d}, \textbf{t})}) \right|_{\frac{2N}{N+4}}= O\left( \varepsilon |\ln \varepsilon| \right).
		\]
		
	\end{proof}

\bigskip\bigskip
	
%	\section{Expancion de la parte no lineal del funcional de energia reducido}

\section{Expansion  of the reduced energy functional}\label{sect:Reduced-Energy}

In this section, we estimate  the nonlinear  term  of the reduced energy functional.  We will focus on the estimations in the case \eqref{typ2} as the expansion of the nonlinear part in case \eqref{typ1} is simpler and can be directly derived from the former. 
So, let $(\zeta^{0}, \textbf{d}, \textbf{t})\in \Sigma_{2}$ and, as before, let \begin{equation*}
	\eta=\min\left\{d(\xi_{1}, \partial\Omega), d(\xi_{2}, \partial\Omega), \dfrac{|\xi_{1}-\xi_{2}|}{2} \right\}=\min\left\{t_{1}\varepsilon,t_{2}\varepsilon,\dfrac{\varepsilon|t_{1}-t_{2}|}{2} \right\}
\end{equation*}
\medskip
    
	\begin{lemma}\label{nonlinearij}
		The following estimate holds true:
		\begin{equation}
			\begin{aligned}
				&\frac{1}{p-\varepsilon} \int_{\Omega} a(x)|PU_{\delta_{1},\xi_{1}}-PU_{\delta_{2},\xi_{2}}|^{p-\varepsilon}dx = \frac{1}{p} \int_{\Omega} a(x) |	V_{(\zeta^{0},\textbf{d}, \textbf{t})}
				|^p \, dx \\
				&+ \varepsilon \left[ \frac{1}{p^2} \int_{\Omega} a(x) |	V_{(\zeta^{0},\textbf{d}, \textbf{t})}
				|^{p} dx - \frac{1}{p} \int_{\Omega} a(x) |	V_{(\zeta^{0},\textbf{d}, \textbf{t})}
				|^{p-1} \log |	V_{(\zeta^{0},\textbf{d}, \textbf{t})}
				| \, dx \right] \\
				&  + o(\varepsilon) \\
				=& 2a(\zeta^{0})\frac{\gamma_{1}}{p}+\varepsilon2a(\zeta^{0})\frac{\gamma_{1}}{p^{2}}-\varepsilon2a(\zeta^{0})\frac{\gamma_{3}}{p}+\varepsilon\log (\varepsilon)2\frac{N-3}{2}a(\zeta^{0})\frac{\gamma_{1}}{p} \\
				+& \varepsilon \nabla a(\zeta^{0})\cdot\nu(\zeta^{0})(t_{1}+t_{2})\frac{\gamma_{1}}{p}-\varepsilon a(\zeta^{0})\left( \left( \frac{d_{1}}{2t_{1}}\right) ^{N-4}+\left( \frac{d_{1}}{2t_{1}}\right) ^{N-4}\right)\gamma_{2} \\
				-&\varepsilon2 a(\zeta^{0})(d_{1}d_{2})^{\frac{N-4}{2}}\left(\frac{1}{|t_{1}-t_{2}|^{N-4}}-\frac{1}{|t_{1}+t_{2}|^{N-4}} \right) \gamma_{2}\\
				+&\varepsilon\frac{N-4}{2}a(\zeta^{0})\log (d_{1}+d_{2})\frac{\gamma_{1}}{p} +o(\varepsilon).
			\end{aligned}
		\end{equation}
		
	\end{lemma}
	
	\begin{proof}
		\bigskip
		
		The proof of the Lemma follows directly from estimates:
		\begin{equation*}
			\begin{aligned}
				\frac{1}{p}\int_{\Omega}a(x)|PU_{\delta_{1},\xi_{1}}-PU_{\delta_{2},\xi_{2}}|^{p}\,dx=&2a(\zeta^{0})\frac{\gamma_{1}}{p}+\varepsilon(t_{1}+t_{2}) \nabla a(\zeta^{0})\cdot\nu(\zeta^{0})\frac{\gamma_{1}}{p}\\
				-&a(\zeta^{0})\varepsilon\left( \left( \frac{d_{1}}{2t_{1}}\right) ^{N-4}+\left( \frac{d_{1}}{2t_{1}}\right) ^{N-4}\right)\gamma_{2} \\
				-&2 a(\zeta^{0})\varepsilon(d_{1}d_{2})^{\frac{N-4}{2}}\left(\frac{1}{|t_{1}-t_{2}|^{N-4}}-\frac{1}{|t_{1}+t_{2}|^{N-4}} \right) \gamma_{2}+ o(\varepsilon),\\
				\frac{\varepsilon}{p^{2}}	\int_{\Omega}a(x)|PU_{\delta_{1},\xi_{1}}-PU_{\delta_{2},\xi_{2}}|^{p}\,dx=&\varepsilon2a(\zeta^{0})\frac{\gamma_{1}}{p^{2}}+o(\varepsilon),
			\end{aligned}
		\end{equation*}
		
		\medskip
		
		and
		\begin{equation*}
			\begin{aligned}
				-\frac{\varepsilon}{p}\int_{\Omega}& a(x)\left|PU_{\delta_{1},\xi_{1}}-PU_{\delta_{2},\xi_{2}}\right|^{p}\log\left(\left|PU_{\delta_{1},\xi_{1}}-PU_{\delta_{2},\xi_{2}}\right|\right)\,dx\\
				=&-\varepsilon2a(\zeta^{0})\frac{\gamma_{3}}{p}+\varepsilon\frac{N-4}{2}a(\zeta^{0})\log (d_{1}+d_{2})\frac{\gamma_{1}}{p}+\varepsilon\log (\varepsilon)(N-3)a(\zeta^{0})\frac{\gamma_{1}}{p} +o(\varepsilon),
			\end{aligned}
		\end{equation*}
		
		which are derived from Lemmas \ref{LemaP1} and \ref{LemaP2}.
	\end{proof}

	\bigskip\bigskip

	\begin{lemma}\label{LemaP1}
		The following estimate holds true:
		\begin{equation}
			\begin{aligned}
				\int_{\Omega}a(x)|PU_{\delta_{1},\xi_{1}}-PU_{\delta_{2},\xi_{2}}|^{p}\,dx=&2a(\zeta^{0})\gamma_{1}+\varepsilon(t_{1}+t_{2}) \nabla a(\zeta^{0})\cdot\nu(\zeta^{0})\gamma_{1}\\
				-&pa(\zeta^{0})\varepsilon\left( \left( \frac{d_{1}}{2t_{1}}\right) ^{N-4}+\left( \frac{d_{1}}{2t_{1}}\right) ^{N-4}\right)\gamma_{2} \\
				-&2p a(\zeta^{0})\varepsilon(d_{1}d_{2})^{\frac{N-4}{2}}\left(\frac{1}{|t_{1}-t_{2}|^{N-4}}-\frac{1}{|t_{1}+t_{2}|^{N-4}} \right) \gamma_{2}+ o(\varepsilon).
			\end{aligned}
		\end{equation}
		
	\end{lemma}
	\begin{proof}
		We write:
		\begin{equation*}
			\begin{aligned}
				\int_{\Omega}a(x)|PU_{\delta_{1},\xi_{1}}-PU_{\delta_{2},\xi_{2}}|^{p}\,dx&=\int_{\Omega}a(x)\left( |PU_{\delta_{1},\xi_{1}}-PU_{\delta_{2},\xi_{2}}|^{p}-U_{\delta_{1},\xi_{1}}^{p}-U_{\delta_{2},\xi_{2}}^{p}\right)\,dx \\
				+&\int_{\Omega}a(x)U_{\delta_{1},\xi_{1}}^{p}\,dx+\int_{\Omega}a(x)U_{\delta_{2},\xi_{2}}^{p}\,dx.
			\end{aligned}
		\end{equation*}
		
		\bigskip
		
		From estimation \eqref{tmi1}, it can be directly deduced:
		\begin{equation*}
			\begin{aligned}
				\int_{\Omega}a(x)U_{\delta_{1},\xi_{1}}^{p}+\int_{\Omega}a(x)U_{\delta_{2},\xi_{2}}^{p}\,dx
				&=\sum_{i=1}^{2}\left( a(\zeta^{0})\gamma_{1}+\varepsilon t_{i}\nabla a(\zeta^{0})\cdot\nu(\zeta^{0})\gamma_{1}\right) +o(\varepsilon)\\
				&=2a(\zeta^{0})\gamma_{1}+\varepsilon(t_{1}+t_{2}) \nabla a(\zeta^{0})\cdot\nu(\zeta^{0})\gamma_{1} +o(\varepsilon).
			\end{aligned}
		\end{equation*}
		
		\bigskip
		
		We will devote the rest of the proof to estimating the following term:
		\begin{equation*}
			\begin{aligned}
				&\int_{\Omega}a(x)\left( |PU_{\delta_{1},\xi_{1}}-PU_{\delta_{2},\xi_{2}}|^{p}-U_{\delta_{1},\xi_{1}}^{p}-U_{\delta_{2},\xi_{2}}^{p}\right)\,dx\\
				=&	\int_{B_\eta(\xi_{1})}a(x)\left( |PU_{\delta_{1},\xi_{1}}-PU_{\delta_{2},\xi_{2}}|^{p}-U_{\delta_{1},\xi_{1}}^{p}-U_{\delta_{2},\xi_{2}}^{p}\right)\,dx\\
				+&	\int_{B_\eta(\xi_{2})}a(x)\left( |PU_{\delta_{1},\xi_{1}}-PU_{\delta_{2},\xi_{2}}|^{p}-U_{\delta_{1},\xi_{1}}^{p}-U_{\delta_{2},\xi_{2}}^{p}\right)\,dx\\
				+&	\int_{\Omega\setminus (B_\eta(\xi_{1})\cup B_\eta(\xi_{2})) }a(x)\left( |PU_{\delta_{1},\xi_{1}}-PU_{\delta_{2},\xi_{2}}|^{p}-U_{\delta_{1},\xi_{1}}^{p}-U_{\delta_{2},\xi_{2}}^{p}\right)\,dx.\\
			\end{aligned}
		\end{equation*}
		
		\bigskip
		
		A straightforward calculation shows that:
		\begin{equation*}
			\begin{aligned}
				\int_{\Omega\setminus (B_\eta(\xi_{1})\cup B_\eta(\xi_{2})) }a(x)\left( |PU_{\delta_{1},\xi_{1}}-PU_{\delta_{2},\xi_{2}}|^{p}-U_{\delta_{1},\xi_{1}}^{p}-U_{\delta_{2},\xi_{2}}^{p}\right)\,dx=o(\varepsilon).
			\end{aligned}
		\end{equation*}
		
		On the other hand, using the equations \eqref{tmi2} and \eqref{isj}, we obtain:
		\begin{equation*}
			\begin{aligned}
				&\int_{B_\eta(\xi_{1})}a(x)\left( |PU_{\delta_{1},\xi_{1}}-PU_{\delta_{2},\xi_{2}}|^{p}-U_{\delta_{1},\xi_{1}}^{p}-U_{\delta_{2},\xi_{2}}^{p}\right)\,dx\\
				=&p\int_{B_\eta(\xi_{1})}a(x)U_{\delta_{1},\xi_{1}}^{p-1}\left(\left(PU_{\delta_{1},\xi_{1}}-U_{\delta_{1},\xi_{1}}\right)-PU_{\delta_{2},\xi_{2}} \right)\,dx+ \int_{B_\eta(\xi_{1})}a(x)U_{\delta_{2},\xi_{2}}^{p}\,dx\\
				+&O\left(\int_{B_\eta(\xi_{1})}\left|U_{\delta_{1},\xi_{1}}+\Theta_{x}\left(PU_{\delta_{1},\xi_{1}}-U_{\delta_{1},\xi_{1}}-PU_{\delta_{2},\xi_{2}} \right)  \right|^{p-2}\left(PU_{\delta_{1},\xi_{1}}-U_{\delta_{1},\xi_{1}}-PU_{\delta_{2},\xi_{2}} \right)^{2} \,dx\right) \\
				=&p\left( -a(\zeta^{0})\varepsilon\left( \frac{d_{1}}{2t_{1}}\right) ^{N-4}\gamma_{2}\right) -p\left( a(\zeta^{0})\varepsilon(d_{1}d_{2})^{\frac{N-4}{2}}\left(\frac{1}{|t_{1}-t_{2}|^{N-4}}-\frac{1}{|t_{1}+t_{2}|^{N-4}} \right) \gamma_{2} \right)+ o(\varepsilon).
			\end{aligned}
		\end{equation*}
		Analogously, it follows that
		\begin{equation*}
			\begin{aligned}
				&\int_{B_\eta(\xi_{2})}a(x)\left( |PU_{\delta_{1},\xi_{1}}-PU_{\delta_{2},\xi_{2}}|^{p}-U_{\delta_{1},\xi_{1}}^{p}-U_{\delta_{2},\xi_{2}}^{p}\right)\,dx\\
				=&p\left( -a(\zeta^{0})\varepsilon\left( \frac{d_{2}}{2t_{2}}\right) ^{N-4}\gamma_{2}\right) -p\left( a(\zeta^{0})\varepsilon(d_{1}d_{2})^{\frac{N-4}{2}}\left(\frac{1}{|t_{1}-t_{2}|^{N-4}}-\frac{1}{|t_{1}+t_{2}|^{N-4}} \right) \gamma_{2} \right)+ o(\varepsilon).
			\end{aligned}
		\end{equation*}

		\medskip

		In the last two expressions, we have assumed that, if $1\leq i,j\leq2, i\neq j$, then:
		\begin{equation*}
			\int_{B_\eta(\xi_{i})}\left|U_{\delta_{i},\xi_{i}}+\Theta_{x}\left(PU_{\delta_{i},\xi_{i}}-U_{\delta_{i},\xi_{i}}-PU_{\delta_{j},\xi_{j}} \right)  \right|^{p-2}\left(PU_{\delta_{i},\xi_{i}}-U_{\delta_{i},\xi_{i}}-PU_{\delta_{j},\xi_{j}} \right)^{2} \,dx=o(\varepsilon).
		\end{equation*}
which is a well-known fact, and we will not provide the details here. This concludes the proof.

	\end{proof}
	
	\bigskip\bigskip
	\begin{lemma}\label{LemaP2}
		The following estimate holds true:
		\begin{equation*}
			\begin{aligned}
				&\int_{\Omega} a(x)|PU_{\delta_{1},\xi_{1}}-PU_{\delta_{2},\xi_{2}}|^{p}\log\left(|PU_{\delta_{1},\xi_{1}}-PU_{\delta_{2},\xi_{2}}|\right)\,dx\\
				=&2a(\zeta^{0})\gamma_{3}-\frac{N-4}{2}a(\zeta^{0})\log (d_{1}+d_{2})\gamma_{1}-(N-3)a(\zeta^{0})\log (\varepsilon)\gamma_{1} +o(1),
			\end{aligned}
		\end{equation*}
		where 
		\begin{equation*}
			\gamma_{3}=	\alpha_{N}^{p}\int_{\mathbb{R}^{N}} U_{1,0}^{p}\log (U_{1,0})\,dx.
		\end{equation*}
	\end{lemma}

\begin{proof} We write:
	\begin{equation}\label{logmain1}
		\begin{aligned}
			&\int_{\Omega} a(x)|PU_{\delta_{1},\xi_{1}}-PU_{\delta_{2},\xi_{2}}|^{p}\log(|PU_{\delta_{1},\xi_{1}}-PU_{\delta_{2},\xi_{2}}|)\,dx\\
			=&\int_{B_{\eta(\xi_{1})}} a(x)|PU_{\delta_{1},\xi_{1}}-PU_{\delta_{2},\xi_{2}}|^{p}\log(|PU_{\delta_{1},\xi_{1}}-PU_{\delta_{2},\xi_{2}}|)\,dx\\
			+&\int_{B_{\eta(\xi_{2})}} a(x)|PU_{\delta_{1},\xi_{1}}-PU_{\delta_{2},\xi_{2}}|^{p}\log(|PU_{\delta_{1},\xi_{1}}-PU_{\delta_{2},\xi_{2}}|)\,dx\\
			+&\int_{\Omega\setminus (B_{\eta(\xi_{1})}\cup B_{\eta(\xi_{2})})} a(x)|PU_{\delta_{1},\xi_{1}}-PU_{\delta_{2},\xi_{2}}|^{p}\log(|PU_{\delta_{1},\xi_{1}}-PU_{\delta_{2},\xi_{2}}|)\,dx.
		\end{aligned}
	\end{equation}
	\medskip
	It is easy to see that:
	\begin{equation}\label{logmain11}
		\begin{aligned}
			\int_{\Omega\setminus (B_{\eta(\xi_{1})}\cup B_{\eta(\xi_{2})})} a(x)|PU_{\delta_{1},\xi_{1}}-PU_{\delta_{2},\xi_{2}}|^{p}\log(|PU_{\delta_{1},\xi_{1}}-PU_{\delta_{2},\xi_{2}}|)=o(1).
		\end{aligned}
	\end{equation}
	The remaining two terms in equation \eqref{logmain1} can be addressed similarly. Thus, we will focus on the first term:	
	\begin{equation*}
		\begin{aligned}
			&\int_{B_{\eta(\xi_{1})}} a(x)|PU_{\delta_{1},\xi_{1}}-PU_{\delta_{2},\xi_{2}}|^{p}\log\left(\left|PU_{\delta_{1},\xi_{1}}-PU_{\delta_{2},\xi_{2}}\right|\right)\,dx\\
			=&\int_{B_{\eta(\xi_{1})}} a(x)|PU_{\delta_{1},\xi_{1}}-PU_{\delta_{2},\xi_{2}}|^{p}\log\left(\left|PU_{\delta_{1},\xi_{1}}\left(1-\frac{PU_{\delta_{2},\xi_{2}}}{PU_{\delta_{1},\xi_{1}}}\right)\right|\right)\,dx\\
			=&\int_{B_{\eta(\xi_{1})}} a(x)|PU_{\delta_{1},\xi_{1}}-PU_{\delta_{2},\xi_{2}}|^{p}\log(|PU_{\delta_{1},\xi_{1}}|)\,dx\\
			+&\int_{B_{\eta(\xi_{1})}} a(x)|PU_{\delta_{1},\xi_{1}}-PU_{\delta_{2},\xi_{2}}|^{p}\log\left(\left|1-\frac{PU_{\delta_{2},\xi_{2}}}{PU_{\delta_{1},\xi_{1}}}\right|\right)\,dx\\
			=&\int_{B_{\eta(\xi_{1})}} a(x)|PU_{\delta_{1},\xi_{1}}-PU_{\delta_{2},\xi_{2}}|^{p}\log(|PU_{\delta_{1},\xi_{1}}|)\,dx+o(1).
		\end{aligned}
	\end{equation*}
	The last equality in the previous equation holds because:
	\begin{equation*}
		\begin{aligned}
			\int_{B_{\eta(\xi_{1})}} a(x)|PU_{\delta_{1},\xi_{1}}-PU_{\delta_{2},\xi_{2}}|^{p}\log\left(\left|1-\frac{PU_{\delta_{2},\xi_{2}}}{PU_{\delta_{1},\xi_{1}}}\right|\right)=o(1).\\
		\end{aligned}
	\end{equation*}
	Now we estimate:
	\begin{equation}\label{logmain12}
		\begin{aligned}
			&\int_{B_{\eta(\xi_{1})}} a(x)|PU_{\delta_{1},\xi_{1}}(x)-PU_{\delta_{2},\xi_{2}}(x)|^{p}\log\left|PU_{\delta_{1},\xi_{1}}(x)\right|\,dx\\
			&=\int_{B_{\eta(\xi_{1})} } a(x)|PU_{\delta_{1},\xi_{1}}(x)|^{p}\log|PU_{\delta_{1},\xi_{1}}(x)|\,dx+O\left(\int_{B_{\eta(\xi_{1})}}a(x)|PU_{\delta_{1},\xi_{1}}(x)|^{p-1}(PU_{\delta_{2},\xi_{2}}(x))\log|PU_{\delta_{1},\xi_{1}}(x)|\,dx \right)\\
			&=\int_{B_{\eta(\xi_{1})} } a(x)U_{\delta_{1},\xi_{1}}^{p}(x)\log U_{\delta_{1},\xi_{1}}(x)\,dx+O\left(\int_{B_{\eta(\xi_{1})} } a(x)\left(|PU_{\delta_{1},\xi_{1}}(x)|^{p}\log| PU_{\delta_{1},\xi_{1}}(x)|-U_{\delta_{1},\xi_{1}}^{p}(x)\log U_{\delta_{1},\xi_{1}}(x)\right) \,dx\right) \\
			&+O\left(\int_{B_{\eta(\xi_{1})}}a(x)|PU_{\delta_{1},\xi_{1}}|^{p-1}(PU_{\delta_{2},\xi_{2}}(x))\log|PU_{\delta_{1},\xi_{1}}(x)| \,dx\right)\\
			&=\int_{B_{\eta(\xi_{1})}} a(x)U_{\delta_{1},\xi_{1}}^{p}(x)\log U_{\delta_{1},\xi_{1}}(x)\,dx+o(1)\\
			&=-\frac{N-4}{2}a(\zeta^{0})\log (d_{1})\gamma_{1}-\frac{N-3}{2}a(\zeta^{0})\log (\varepsilon)\gamma_{1}+a(\zeta^{0})\gamma_{3} +o(1).
		\end{aligned}
	\end{equation}
	In the last equality,  we made use of the following estimate:
	\begin{equation*}
		\begin{aligned}
			&\int_{B_{\eta(\xi_{1})}} a(x)U_{\delta_{1},\xi_{1}}^{p}(x)\log U_{\delta_{1},\xi_{1}}(x)\,dx=\alpha_{N}^{p}\int_{B_{\eta(\xi_{1})} } a(x)\left(\frac{\delta_{1}^{N}}{(\delta_{1}^{2}+|x-\xi_{1}|^{2})^{N}} \right) \log \left(\frac{\delta_{1}^{\frac{N-4}{2}}}{(\delta_{1}^{2}+|x-\xi_{1}|^{2})^{\frac{N-4}{2}}}\right)\,dx\\
			&=\alpha_{N}^{p}\int_{B_{\frac{\eta}{\delta_{1}}}(0) } a(\delta_{1}y+\xi_{1})\left(\frac{1}{(1+|y|^{2})^{N}} \right) \log \left(\frac{\delta_{1}^{\frac{N-4}{2}}}{(\delta_{1}^{2}+|\delta_{1}y|^{2})^{\frac{N-4}{2}}}\right)\,dy\\
			&=\alpha_{N}^{p}\int_{\frac{B_{1}-\xi_{1}}{\delta_{1}} } a(\xi_{1}^{0})\left(\frac{1}{(1+|y|^{2})^{N}} \right)\left( \log \left(\delta_{1}^{-\frac{N-4}{2}}\right)+ \log \left(U_{1,0}(x)\right)\right) \,dy +o(1)\\
			&=-\frac{N-4}{2}a(\zeta^{0})\log (\delta_{1})\gamma_{1}+a(\xi_{1}^{0})\gamma_{3} +o(1)\\
			&=-\frac{N-4}{2}a(\zeta^{0})\log (d_{1}\varepsilon^{\frac{N-3}{N-4}})\gamma_{1}+a(\zeta^{0})\gamma_{3} +o(1)\\
			&=-\frac{N-4}{2}a(\zeta^{0})\log (d_{1})\gamma_{1}-\frac{N-3}{2}a(\zeta^{0})\log (\varepsilon)\gamma_{1}+a(\zeta^{0})\gamma_{3} +o(1),
		\end{aligned}
	\end{equation*}
and the fact that, with direct calculations, we have that:
	\begin{equation*}
		\begin{aligned}
			O\left(\int_{B_{\eta(\xi_{1})} } a(x)\left(|PU_{\delta_{1},\xi_{1}}(x)|^{p}\log| PU_{\delta_{1},\xi_{1}}(x)|-U_{\delta_{1},\xi_{1}}^{p}(x)\log U_{\delta_{1},\xi_{1}}(x)\right) \,dx\right)&=o(1)\\
			O\left(\int_{B_{\eta(\xi_{1})}}a(x)|PU_{\delta_{1},\xi_{1}}(x)|^{p-1}\left(PU_{\delta_{2},\xi_{2}}(x)\right)\log|PU_{\delta_{1},\xi_{1}}(x)| \,dx\right)&=o(1).
		\end{aligned}
	\end{equation*}

	\medskip

	Similarly, we obtain the estimate of the second term in equation \eqref{logmain1}: 
	\begin{equation}\label{logmain13}
		\begin{aligned}
			&\int_{B_{\eta(\xi_{2})}} a(x)\left|PU_{\delta_{1},\xi_{1}}(x)-PU_{\delta_{2},\xi_{2}}(x)\right|^{p}\log(|PU_{\delta_{1},\xi_{1}}(x)-PU_{\delta_{2},\xi_{2}}(x)|) \,dx\\
			=&-\frac{N-4}{2}a(\zeta^{0})\log (d_{2})\gamma_{1}-\frac{N-3}{2}a(\zeta^{0})\log (\varepsilon)\gamma_{1}+a(\zeta^{0})\gamma_{3} +o(1).
		\end{aligned}
	\end{equation}
	
	The proof is completed by taking into account equations \eqref{logmain1}, \eqref{logmain11}, \eqref{logmain12}, and \eqref{logmain13}.
	\end{proof}

\bigskip

This completes the proof of the estimations for the nonlinear part of the reduced energy functional in case \eqref{typ2}. The estimation for the nonlinear part in case \eqref{typ1} is given by:

	\begin{lemma}
		\label{lemma:NL}
	Let $(\boldsymbol{\xi},\textbf{d}, \textbf{t})\in\Sigma_{1}$, then	the following estimate holds true:
		\begin{equation}
			\begin{aligned}
				\frac{1}{p-\varepsilon} \int_{\Omega} a(x) |	V_{(\boldsymbol{\xi},\textbf{d}, \textbf{t})}
				(x)|^{p-\varepsilon} \, dx &= \frac{1}{p} \int_{\Omega} a(x) |	V_{(\boldsymbol{\xi},\textbf{d}, \textbf{t})}
				(x)|^p \, dx \\
				&\quad + \varepsilon \left[ \frac{1}{p^2} \int_{\Omega} a(x) |	V_{(\boldsymbol{\xi},\textbf{d}, \textbf{t})}
				(x)|^{p} dx - \frac{1}{p} \int_{\Omega} a(x) |	V_{(\boldsymbol{\xi},\textbf{d}, \textbf{t})}
				(x)|^{p-1} \log |	V_{(\boldsymbol{\xi},\textbf{d}, \textbf{t})}
				(x)| \, dx \right] \\
				&\quad + o(\varepsilon) \\
				&= \sum_{i=1}^{k}a(\xi_{i}^{0})\frac{\gamma_{1}}{p}\\
				& + \sum_{i=1}^{k}\left(a(\xi_{i}^{0})\frac{\gamma_{1}}{p^{2}}+ t_{i}\nabla a(\xi_{i}^{0})\cdot\nu(\xi_{i}^{0})\frac{\gamma_{1}}{p}-\left( \frac{d_{i}}{2t_{i}}\right) ^{N-4}a(\xi_{i}^{0})\gamma_{2}-a(\xi_{i}^{0})\frac{\gamma_{3}}{p} \right) \varepsilon\\
				&+\sum_{i=1}^{k}\left(  \frac{N-4}{2}a(\xi_{i}^{0})\log (d_{i})\frac{\gamma_{1}}{p}+\frac{N-3}{2}a(\xi_{i}^{0})\log(\varepsilon)\frac{\gamma_{1}}{p}\right)\varepsilon +o(\varepsilon).
			\end{aligned}
		\end{equation}
		
	\end{lemma}

\bibliographystyle{plain}

\end{document}